\documentclass[11pt,leqno]{amsart}
\usepackage{amsmath}
\usepackage{amsfonts}
\usepackage{amssymb}
\usepackage{graphicx}
\usepackage{hyperref}
\usepackage{mathrsfs} % fancy calligraphic style
\usepackage{relsize} % big size of integrals
\usepackage{xcolor}
\usepackage{color}
\theoremstyle{plain}
\newtheorem{theorem}{Theorem}
\newtheorem{ansatz}[theorem]{Ansatz}

\newtheorem{corollary}[theorem]{Corollary}
\newtheorem{definition}[theorem]{Definition}
\newtheorem{lemma}[theorem]{Lemma}
\newtheorem{problem}{Problem}
\newtheorem{proposition}[theorem]{Proposition}

\numberwithin{equation}{section}
\numberwithin{theorem}{section}

\newtheorem{theoremA}{Theorem}

\theoremstyle{definition}

\newtheorem{remark}[theorem]{Remark}

\newcommand{\vol}{\mathrm{vol}}

\renewcommand{\div}{\operatorname{div}}

\def\XXint#1#2#3{{\setbox0=\hbox{$#1{#2#3}{\int}$}
     \vcenter{\hbox{$#2#3$}}\kern-.5\wd0}}

\newcommand{\rr}{\mathbb{R}}
\renewcommand{\ss}{\mathbb{S}}

\newcommand{\nn}{\mathbb{N}}

\newcommand{\mm}{\mathbb{M}}

\newcommand{\sect}{\operatorname{Sect}}
\newcommand{\ric}{\operatorname{Ric}}

\renewcommand{\a}{\alpha}
\renewcommand{\b}{\beta}

\renewcommand{\d}{\delta}
\newcommand{\e}{\varepsilon}
\renewcommand{\l}{\lambda}

\newcommand{\s}{\sigma}
\newcommand{\vp}{\varphi}

\renewcommand{\O}{\Omega}

\newcommand{\CO}{\mathcal{O}}

\renewcommand{\S}{\mathscr{S}}
\newcommand{\SR}{\mathscr{S\!R}}

\newcommand{\re}{\textrm{e}}

\begin{document}

%%% front matter %%%

\title[Qualitative properties of subsolutions of nonlinear PDEs]{Qualitative properties of bounded subsolutions of nonlinear PDEs}
\author {Davide Bianchi}
\address{Dipartimento di Fisica e Matematica\\
Universit\`a dell'Insubria - Como\\
via Valleggio 11\\
I-22100 Como, ITALY}
\email{d.bianchi9@uninsubria.it}

\author{Stefano Pigola}
\address{Dipartimento di Fisica e Matematica\\
Universit\`a dell'Insubria - Como\\
via Valleggio 11\\
I-22100 Como, ITALY}
\email{stefano.pigola@uninsubria.it}

\author{Alberto G. Setti}
\address{Dipartimento di Fisica e Matematica
\\
Universit\`a dell'Insubria - Como\\
via Valleggio 11\\
I-22100 Como, ITALY}
\email{alberto.setti@uninsubria.it}

\begin{abstract}
 We study decay and compact support properties of positive and bounded solutions of $\Delta_{p} u \geq \Lambda(u)$ on the exterior of a compact set of a complete manifold with rotationally symmetry.  In the same setting, we also give a new characterization of stochastic completeness for the $p$-Laplacian in terms of a global $W^{1,p}$-regularity of such solutions. One of the tools we use is a nonlinear version of the Feller property which we investigate on general Riemannian manifolds and which we establish  under integral Ricci curvature conditions.
\end{abstract}

\date{January 15, 2020}
\maketitle
%\tableofcontents

\section{Introduction and main results}
This paper is devoted to the study of global properties of solutions at infinity of a certain class of  PDEs involving the $p$-Laplacian on a complete Riemannian manifold $(M,g)$.\smallskip

Recall that, for $1<p<+\infty$, the $p$-Laplacian is the nonlinear operator acting on $W^{1,p}_{\mathrm{loc}}$-functions as follows
\[
\Delta_{p}u = \div(|\nabla u|^{p-2}\nabla u),
\]
where the divergence has to be understood in the distributional sense. Obviously, when $p=2$, $\Delta = \Delta_{2}$ is the (negative definite) Laplace-Beltrami operator of $(M,g)$.

We are interested in differential equations of the form
\begin{equation}\label{intro-eq1}
\Delta_{p}u = \Lambda\left(  u\right)  ,\text{ on }M \setminus \bar \O,
\end{equation}
where $\Omega\Subset M$ is a domain and $\Lambda:[0,+\infty)\rightarrow
\lbrack0,+\infty)$ is a non-decreasing function satisfying the following structural conditions:
\begin{equation}\label{intro-eq2}
\text{(a) }\Lambda\left(  0\right)  =0\text{; (b) }\Lambda\left(  t\right)
>0\text{, }\forall t>0\text{; (c) }\liminf_{t\rightarrow0+}\frac
{\Lambda\left(  t\right)  }{t^{\xi}}>0,
\end{equation}
for some $0\leq\xi\leq p-1$. We say that $v$ is a {\it subsolution} of \eqref{intro-eq1} if $\Delta_{p} v \geq \Lambda(v)$, namely, the following inequality
\[
-\int_{M \setminus \bar \O} g(|\nabla v|^{p-2}\nabla v,\nabla \vp) \geq \int_{M \setminus \bar \O} \Lambda(v) \vp
\]
holds for every $0 \leq \vp \in C^{\infty}_{c}(M\setminus \bar \O)$. We also call $v$ a {\it supersolution} if the opposite inequality holds.\smallskip

The qualitative study of solutions of \eqref{intro-eq1} will be carried out on a special class of Riemannian manifolds whose metric has a rotational symmetry.  Given a smooth warping function $\s:[0,+\infty) \to [0,+\infty)$ satisfying
\[
(i)\, \s(t)>0, \,\, \forall t>0;\quad (ii)\, \s^{(2k)}(0) = 0,\,\, \forall k\in \nn; \quad (iii)\, \s'(0)=1,
\]
the {\it model manifold} $\mm^{m}_{\s}$ is defined as  Euclidean space $\rr^{m}$ endowed with the smooth, complete Riemannian metric that in polar coordinates is given by
\[
g = dr \otimes dr + \s^{2}(r) g_{\ss^{m-1}}.
\]
The origin $0 \in \rr^{m}$ is a pole of $\mm^{m}_{\s}$ in the sense that the corresponding exponential map is a global diffeomorphism. Moreover, the $r$-coordinate is nothing but the distance function from $0$.\smallskip

In this setting, one of our main contributions is the following decay property of positive and bounded subsolutions of \eqref{intro-eq1}.

\begin{theoremA}\label{th-main1}
Let $\mm^{m}_{\s}$ be a complete model manifold satisfying $\ric \geq -C$, for some constant $C \geq 0$. Let $0 \leq u\in C^{0}(\mm^{m}_{\s}\setminus \Omega)\cap W^{1,p}_{\mathrm{loc}}(\mm^{m}_{\s}\setminus \Omega) $ be a bounded subsolution of \eqref{intro-eq1} with $\Lambda$ satisfying \eqref{intro-eq2} for some $0 \leq \xi \leq p-1$. Then
\[
\lim_{x\rightarrow\infty}u\left(  x\right)  =0.
\]
\end{theoremA}

The strategy to prove Theorem \ref{th-main1} is inspired by \cite{BPS-Spectral} and   is based on the following two ingredients of independent interest:
\begin{itemize}
 \item [a)] A suitable version for the $p$-Laplacian of the so called {\it Feller property} which, in the linear setting, was systematically studied in \cite{Az, PS-Feller};
 \item [b)] A global comparison theory for positive and bounded sub- supersolutions of \eqref{intro-eq1} with $\Lambda(t) = \lambda t^{p-1}$ that, for linear operators, was introduced in \cite{PS-Feller, BPS-Spectral}.
\end{itemize}
Some of the crucial arguments used in \cite{Az, PS-Feller, BPS-Spectral} rely either on the linearity of the (drifted) Laplacian or on heat kernel properties of the underlying manifolds related to the Feller property. Therefore,   to carry out our nonlinear program  new ideas have to be introduced from the very beginning of the theory. In particular, the $p$-Laplace version of the $L^{\infty}$-comparison alluded to in b) looks really challenging and, so far, we have been able to obtain it in the special setting of model manifolds by passing through the $L^{p}$-theory. This is precisely the point where rotational symmetry enters the game. On the other hand,  this investigation has led to the discovery of an interesting connection  between the {\it stochastic completeness} of the $p$-Laplacian and a global $W^{1,p}$-regularity property of subsolutions of \eqref{intro-eq1} which appears to be new even in the linear case. Recall that $(M,g)$ is called {\it stochastically complete} if  the Brownian motion trajectories on $M$  do not explode in finite time almost surely. It was proved by R. Azencott that this is equivalent to the fact that, for any $\l>0$, the only bounded solution $u \geq 0$ of the equation $\Delta u = \l  u$ is the trivial solution $u \equiv 0$.
The notion of stochastic completeness extends to the $p$-realm by using this Liouville-type viewpoint; see \cite{PRS-revista, MV-tams} and also the very recent \cite{MP} where, using the Harvey-Lawson approach, viscosity solutions of fully nonlinear operators are included in the picture. With this terminology, we have the following result.

\begin{theoremA}\label{th-main2}
Let $1<p<+\infty$ be fixed. The complete model manifold $\mm^{m}_{\s}$ is $p$-stochastically complete if and only if, given a  domain $\O \Subset \mm^{m}_{\s}$ and a real number $\l >0$, every solution $0 \leq u \in W^{1,p}_{\mathrm{loc}}(\mm^{m}_{\s}\setminus \bar \O) \cap C^{0}(\mm^{m}_{\s} \setminus \O)$ of the problem
\[
\begin{cases}
\Delta_{p} u \geq \l u^{p-1}  & \,\mm^{m}_{\s} \setminus \bar \O \\
\sup_{\mm^{m}_{\s}\setminus \O} u <+\infty
\end{cases}
\]
satisfies
\[
u \in W^{1,p}(\mm^{m}_{\s} \setminus   \bar \O_{\e}),
\]
where $\O_{\e}$ denotes any $\e$-neighborhood of $\O$.\\
Moreover, if $p=2$, the equivalence holds true with $\e= 0$ and without the rotational symmetry assumption on the manifold.
\end{theoremA}

We end this introduction by noting that the decay properties established in the paper have interesting consequences as soon as the geometric conditions on the space are suitably strengthen. Among them, we point out the following result.

\begin{theoremA}\label{th-main3}
 Let $ \mm^{m}_{\s}$ be a complete model manifold whose sectional curvature satisfies $-\kappa^{2}\leq \sect \leq 0$, for some $\kappa \in \rr$. If $u \in C^{1}(\mm^{m}_{\s}\setminus \O)$ is a nonnegative and bounded subsolution of \eqref{intro-eq1} with $\Lambda$ satisfying \eqref{intro-eq2} for some $0 \leq \xi < p-1$, then $u$ has compact support.
\end{theoremA}

In recent years, the study of compact support properties of subsolutions of nonlinear PDEs along the ends of a complete Riemannian manifold has received a great impulse. For an up-to-date account and a comprehensive treatment on the so called {\it compact support principle} and its interplays with various forms of the maximum principles, we refer the reader to \cite{BMPR}. One of the interesting questions raised in that paper concerns the study of conditions under which a suitable nonlinear Feller property implies the compact support principle for solutions at infinity of $\Delta_{p} u \geq f(u)$ (possibly corrected with gradient terms) provided $F^{-1/p} \in L^{1}(0+)$ where $F(t) = \int_{0}^{t} f(s) ds$; see \cite[Problem 5]{BMPR}. Theorem \ref{th-main3} can be considered as a first step in this direction: as we shall see in Section \ref{section-compact}, it is a straightforward consequence of a more general result where the curvature conditions are replaced by the validity of $p$-versions of the stochastic completeness and the Feller property.
\\

The paper is organized as follows:\smallskip

\noindent - In Section \ref{section-pfeller} we extend to the $p$-Laplacian the usual notion of Feller property and we will prove that it holds assuming that the Ricci curvature is bounded below in a suitable integral sense. This appear to be new even in the case of the usual Feller property where $p=2$.  We also provide a necessary and sufficient condition for its validity on a model manifold.\smallskip

\noindent - In Section \ref{section-global-generic} we consider the minimal positive solution of $\Delta_{p} h = \l  h^{p-1}$ on an exterior domain of a generic complete manifold, we establish its (weighted) global Sobolev regularity and we partially prove how this regularity can be inherited by bounded subsolutions of \eqref{intro-eq1} via an $L^{p}$ comparison.\smallskip

\noindent - Section \ref{section-rotational} contains the proof of (more general versions of) Theorem \ref{th-main1} and Theorem \ref{th-main2}. More precisely, we focus on model manifolds that are stochastically complete for the $p$-Laplacian and, for any bounded subsolutions of \eqref{intro-eq1} on exterior domains, we prove a general comparison principle with the minimal solution $h$ alluded to above, and the companion global Sobolev regularity property of subsolutions.\smallskip

\noindent - In the final Section \ref{section-compact}, as an application of the theory developed in the previous sections, we provide a proof of (a generalized version of) Theorem \ref{th-main3}.

%%%%%%%%%%%%%%%%%%%%%%%%%%%%%%%%%%%%%%%%%%%%%%%%%%%%%%%%%%%%%%%%%%%%%%%%%%%%%%

\section{The $p$-Feller property}\label{section-pfeller}

%%%%%%%%%%%%%%%%%%%%%%%%%%%%%%%%%%%%%%%%%%%%%%%%%%%

\subsection{Basic facts}
Let $(M,g)$ be any complete Riemannian manifold with Laplace-Beltrami operator $\Delta$. The heat semigroup $S(t) = e^{-t\Delta}$ is said to satisfy the {\it Feller property} if it preserves the space $C_{0}(M) = \{ u \in C^{0}(M): u(x) \to 0 \text{ as }x\to \infty\}$. It was proved by R. Azencott, \cite{Az}, that this is equivalent to requiring that, for any smooth domain $\O \Subset M$ and for any $\l>0$ the (necessarily unique) minimal solution of the problem
\[
\begin{cases}
 \Delta h = \l h & M \setminus \bar \O \\
 h=1 & \partial \O \\
 h \geq 0 & M \setminus \O
\end{cases}
\]
satisfies $h(x) \to 0$ as $x\to 0$.  We are therefore led to extend the notion of Feller property to the nonlinear setting as follows.
\begin{definition}
Fix $1<p<+\infty$. We say that a Riemannian manifold $(M,g)$ is $p$-Feller if  for any smooth domain $\O \Subset M$, and for any choice of the constant $\l>0$, the minimal  solution $h \in C^{1}(M\setminus \O)$ of the  exterior boundary value problem
\begin{equation} \label{p-feller}
\begin{cases}
\Delta_{p}h=\lambda h^{p-1} & \text{on } M\setminus \bar \O\\
h=1 & \text{on }\partial\O\\
h \geq 0 & M \setminus \O
\end{cases}
\end{equation}
satisfies $h(x) \to 0$ as $x \to \infty$.
\end{definition}
Obviously, the first problem we have to address is whether this definition is well posed. A rather complete answer is contained in the next
\begin{theorem}\label{th-basic}
 Let $(M,g)$ be a (possibly incomplete) Riemannian manifold and let $1<p<+\infty$ be fixed.
\begin{itemize}
 \item [(a)] For any smooth domain $\O \Subset M$ and for any $\l>0$ the minimal solution $h \in C^{1}(M\setminus \O)$ of \eqref{p-feller} exists and satisfies $0 < h <1$ on $M \setminus \bar \O$.
 \item [(b)] Let $\O \Subset M$ be a fixed smooth domain. Given $\l,\tilde \l >0$, let $h,\tilde h$ denote the minimal solutions of \eqref{p-feller} corresponding to these constants. Then $h(x) \to 0$ as $x \to \infty$ if and only if $\tilde h(x) \to 0$ as $x\to \infty$.
 \item [(c)] Let $\l >0$ be a fixed constant. Given smooth domains $\tilde \O, \O \Subset M$
 denote by $\tilde h, h$ the minimal solutions of the corresponding exterior problems \eqref{p-feller}. Then $h(x) \to 0$ as $x\to \infty$ if and only if  $\tilde h(x) \to 0$ as $x \to \infty$.
\end{itemize}
\end{theorem}
\begin{remark}
Traditionally, the proofs of  statements (b) and (c) for the Laplace operator rely on the semigroup formulation of the Feller property. Therefore, new arguments are needed when $p \not=2$.
\end{remark}
\begin{proof}
(a) Let $\{ \O_{n}\}$ be a smooth, relatively compact exhaustion of $M$ with $\O_{0} = \O$. We shall produce a corresponding sequence of solutions of the boundary value problem
\begin{equation}\label{h_n}
\left\{
\begin{array}
[c]{ll}
\Delta_{p}h_n=\lambda h_n\left\vert h_n\right\vert ^{p-2} & \text{on } \Omega_n \setminus \bar  \Omega\\
h_n=1 & \text{on }\partial\Omega\\
h_n=0 & \text{on }\partial\Omega_n.
\end{array}
\right.
\end{equation}
and prove that it converges nicely to the desired minimal solution  of \eqref{p-feller}.

In order to find a weak solution of \eqref{h_n}, the most natural way is to apply the sub/supersolution method within the direct calculus of variations. We thus consider the  functional $E_{p,\l} : \S_{n} \to \rr_{\geq 0}$
\[
E_{p,\l}(u)=\frac{1}{p}\int_{\Omega_n \setminus \bar {\Omega}} |\nabla u |^p +\frac{\lambda}{p} \int_{\Omega_n \setminus \bar {\Omega}} |u|^p
\]
which is coercive and weakly lower semicontinuous on the subset $\S_{n}$ of $W^{1,p}(\Omega_n \setminus \bar {\Omega})$ defined as follows:
\[
\S_{n} = \{ u \in W^{1,p}(\Omega_n \setminus \bar {\Omega}): 0\leq u \leq 1\text{ a.e. },\, u=1\text{ on }\partial \O,\, u=0 \text{ on } \partial \O_{n} \},
\]
where the boundary data are understood in  trace sense. Since the constant functions $\underline{u}\equiv 0,\bar{u}\equiv 1 \in \S$ are, respectively, a subsolution and a supersolution of \eqref{h_n} it follows from e.g. \cite[Theorem 2.4]{St-book} or \cite[Theorem 37]{Ha-lecturenotes} that  $E_{p,\l}$ attains its minimum $h_{n}$ in $\S_n$ and that $0 \leq h_{n} \leq 1$ is a weak solution of \eqref{h_n}\footnote{alternatively, one can use the Perron method to produce a continuous, hence $C^{1,\a}_{loc}(\bar \Omega_{n} \setminus \Omega)$, solution of the problem and, then, apply the comparison principle stated in Remark \ref{rmk-w1p-comparison}. In this way one deduces that the Perron solution is the minimizer of the functional $E_{p,\l}$ on the space $\{ u \in W^{1,p}(\Omega_{n}\setminus \bar \Omega): u=0 \text{ on }\partial \Omega_{n},u=1 \text{ on }\partial \Omega\}$, where the boundary data are understood in the trace sense.}.

Since $h_n \in W^{1,p} \cap L^{\infty}$, by elliptic regularity up to the boundary,  \cite{To,GaZi, Li-nonlinear}, $h_n \in C^{1,\alpha}(\bar \Omega_n \setminus \Omega)$, where $0<\a<1$ depends on the domain.

Summarizing, for every $n$ we have obtained a $C^{1,\alpha}(\bar \Omega_n \setminus \Omega)$ solution $0 \leq h_n \leq 1$ of the Dirichlet problem

\begin{equation}\label{h_n positive}
\left\{
\begin{array}
[c]{ll}
\Delta_{p}h_n=\lambda h_n^{p-1} & \text{on } \Omega_n \setminus \bar \Omega\\
h_n=1 & \text{on }\partial\Omega\\
h_n=0 & \text{on }\partial\Omega_n.
\end{array}
\right.
\end{equation}
Note that, by Schauder-type estimates up to the boundary due to Lieberman,  \cite{Li-nonlinear}, the $C^{1,\a}$-norm of the sequence $\{ h_{n} \}$ is bounded on compact subsets of $M \setminus \Omega$. Therefore, by using the compactness of the embedding $C^{1,\a}(\bar D)  \hookrightarrow C^{1,\b}(\bar D)$, for any $\a > \b$ and $D \Subset M$, together with a diagonal process, we get  that a subsequence $\{ h_{n'} \}$ converges, $C^{1}$-uniformly on compact subsets, to a solution $h \in C^{1}(M \setminus \Omega)$ of (\ref{p-feller}) satisfying $0< h \leq 1$.

The fact that $h$ is the minimal solution follows by comparing on $\O_{n} \setminus  \O$ a generic solution with $h_n$ and, then, letting $n \to +\infty$. Summarizing, $h$ is the minimal, positive solution of (\ref{p-feller}).\smallskip

(b) Let $\O \Subset M$ be a fixed smooth domain. Choose $ \l, \tilde  \l > 0$ and let $h, \tilde  h  \in C^{1}(M\setminus \O)$ be the minimal, positive solutions of \eqref{p-feller} corresponding to these constants. We assume that $h(x) \to 0$, as $x\to \infty$ and we show that this property is shared by $\tilde  h$.\smallskip

To this end, recall that $h = \lim_n h_n$ and $\tilde  h = \lim_n \tilde  h_n$ where $0 \leq h_n,\tilde  h_n \leq 1$ are solutions of the following problems:
\begin{equation}\label{lambdamu1}
\left\{
\begin{array}
[c]{ll}
\Delta_{p}h_{n}= \l h_{n}^{p-1} & \text{on }\Omega_{n}\setminus
\bar  {\Omega}\\
h_{n}=1 & \text{on } \partial \Omega\\
h_{n}=0 & \text{on } \partial \Omega_n.
\end{array}
\right.
\end{equation}
and,
\begin{equation}\label{lambdamu2}
\left\{
\begin{array}
[c]{ll}
\Delta_{p}\tilde  h_{n}=\tilde  \l \tilde  h_{n}^{p-1} & \text{on }\Omega_{n}\setminus
\bar{\Omega}\\
\tilde  h_{n}=1 & \text{on } \partial \Omega\\
\tilde  h_{n}=0 & \text{on } \partial \Omega_n,
\end{array}
\right.
\end{equation}
where $\O_{n} \nearrow M$ is a smooth, relatively compact exhaustion with $\O_{0} = \O$. Now we distinguish two cases:\smallskip

\noindent {\it Case 1:} Assume that $0<\lambda < \tilde  \l$. Then, from \eqref{lambdamu1} we see that $h_n$ is a supersolution of \eqref{lambdamu2}. By comparison,  $\tilde  h_n \leq h_n$ and taking the limit as $n \to \infty$ we conclude $\tilde  h \leq h$. It follows that $\tilde  h(x) \to 0$ as $x \to \infty$.\smallskip

\noindent {\it Case 2:} Suppose that $0 < \tilde  \l  < \l$. We define
 \[
 \alpha =\left(\frac{\l}{\tilde  \l}\right)^{1/(p-1)}>1
 \]
 and, for every $n$, we consider the function
 \[
 w_n = \tilde  h_n^{\alpha} \in C^{1}(M\setminus \O).
 \]
 A direct computation shows that $w_n$ is a subsolution of \eqref{lambdamu1}. Therefore, by comparison, $w_n \leq h_n$ and, letting $n \to \infty$, we conclude
 \[
 \tilde  h^\alpha \leq h.
 \]
It follows that $\tilde  h(x) \to 0$ as $x \to \infty$.
\smallskip

(c) It suffices to consider the case where $\tilde \O \Subset \O \Subset M$. Keeping the notation in the statement, we consider an exhaustion $\O_{n} \nearrow M$ satisfying $\O_{0} = \tilde \O$ and $\O_{1} = \O$ and, for $n \geq 2$, we let $0 \leq h_{n}, \tilde h_{n} \leq 1$ be the solutions of
\[
\left\{
\begin{array}
[c]{ll}
\Delta_{p}h_{n}= \l h_{n}^{p-1} & \text{on }\Omega_{n}\setminus
\bar {\Omega}\\
h_{n}=1 & \text{on } \partial \Omega\\
h_{n}=0 & \text{on } \partial \Omega_n.
\end{array}
\right.
\]
and
\[
\left\{
\begin{array}
[c]{ll}
\Delta_{p} \tilde h_{n}= \l \tilde h_{n}^{p-1} & \text{on }  \Omega_{n}\setminus
\bar {\tilde \Omega}\\
\tilde h_{n}=1 & \text{on } \partial \tilde \Omega\\
\tilde h_{n}=0 & \text{on } \partial \Omega_n.
\end{array}
\right.
\]
respectively. Since $\tilde h_{n} <1 = h_{n}$ on $\partial (\O_{n}\setminus \bar \Omega)$, by comparison we have
\[
\tilde h_{n} \leq h_{n},\quad \text{on } \O_{n} \setminus \O.
\]
Whence, by taking the limit as $n \to +\infty$ we deduce the validity of the pointwise estimate
\begin{equation}\label{upper}
\tilde h \leq h,\quad \text{on } M \setminus \O.
\end{equation}
On the other hand, for every $n \geq 2$, we define

\[
\tilde {\tilde h}_{n} = \frac{\tilde h_{n}}{\min_{\partial \O} \tilde h_{n}}
\]
and we observe that, by homogeneity,
\[
\left\{
\begin{array}
[c]{ll}
\Delta_{p} \tilde {\tilde h}_{n}= \l \tilde {\tilde h}_{n}^{p-1} & \text{on }  \Omega_{n}\setminus
\bar { \Omega}\\
\tilde {\tilde h}_{n} \geq 1 & \text{on } \partial \Omega\\
\tilde {\tilde h}_{n}=0 & \text{on } \partial \Omega_n.
\end{array}
\right.
\]
Then a comparison argument  shows that
\[
h_{n} \leq \tilde{\tilde h}_{n}\quad \text{on }\O_{n} \setminus \O
\]
i.e.
\[
\min_{\partial \O} \tilde h_{n} \cdot h_{n} \leq \tilde{h}_{n}\quad \text{on }\O_{n} \setminus \O.
\]
Since $\tilde {h}_{n}$ converges to $\tilde {h}$ uniformly on compact subsets of $M \setminus \tilde \O$, by taking the limit as $n \to +\infty$ we conclude the validity of the pointwise estimate
\begin{equation}\label{lower}
\min_{\partial \O} \tilde{h} \cdot h \leq \tilde{h}\quad \text{on }M \setminus \O.
\end{equation}
From \eqref{upper} and \eqref{lower} it follows that $h(x)$ decays to $0$ as $x\to \infty$ if and only if so does $\tilde h$.

The proof of the Theorem is completed.
\end{proof}

%%%%%%%%%%%%%%%%%%%%%%%%%%%%%%%%%%%%%%%%%%%%%%%%%%%

\subsection{A characterization on model manifolds}\label{section-feller-rotational}

It was proved in \cite[Theorem 3.4]{PS-Feller} that the validity of the Feller property on $\mm^{m}_{\s}$ is completely characterized by certain integral conditions satisfied by the warping function $\s$. These, in turn, can be read in terms of volume properties of the model. The following result extends this characterization to the $p$-Laplacian.

\begin{theorem}\label{th-pFeller-models}
Let $\mm^{m}_{\s}$  be a complete model manifold with warping function $\s$ and let $1<p<+\infty$ be fixed. Then $\mm^{m}_{\s}$ is $p$-Feller if and only if either
\begin{equation}\label{model-nonparab}
 \frac{1}{\s^{\frac{m-1}{p-1}}} \in L^{1}(+\infty)
\end{equation}
or
\begin{equation}\label{model-parab}
(i)\,  \frac{1}{\s^{\frac{m-1}{p-1}}} \not\in L^{1}(+\infty) \quad \text{and}  \quad
(ii)\, \left\{ \frac{\int_{r}^{+\infty}\s^{m-1}(t)dt}{ \s^{m-1}(r)} \right\}^{\frac{1}{p-1}} \not\in L^{1}(+\infty)
\end{equation}
\end{theorem}

\begin{remark}
 It is well known that \eqref{model-nonparab} is a necessary and sufficient condition for $\mm^{m}_{\s}$ to be {\it  $p$-hyperbolic}; see e.g. \cite[Corollary 5.2]{Tr}. This means that the manifold at hand supports a nonconstant, positive, $p$-superharmonic function. In the rotationally symmetric case, such a function is easily obtained by taking $u(r(x)) = \min \left(1,\int_{r(x)}^{+\infty}\s^{-\frac{m-1}{p-1}}(t) dt\right)$.
\end{remark}

Inspection of the proof of \cite[Theorem 3.4]{PS-Feller} shows that the only delicate issue is the radial symmetry of the minimal solution $h$ of the exterior problem \eqref{p-feller}. Granted this, the arguments used in \cite[Theorem 3.4]{PS-Feller} can be reproduced almost verbatim in the context of the $p$-Laplacian. Now, due to the nonlinearity of the $p$-Laplacian, it is no longer true, as in \cite[Theorem 3.2]{PS-Feller}, that the radialization of $h$ satisfies the same differential inequality. Thus, in order to insure that $h$  inherits the radial symmetry enjoyed by the metric, the domain, the differential operator and the boundary data, we have to follow a different way.

\begin{theorem}\label{lemma-symmetric}
Let $\mm^{m}_{\s}$ be a complete model manifold and, having fixed $R>0$, let $h\in C^{1}(\mm^{m}_{\s} \setminus B_{R}(o))$ be the minimal solution of the problem
\[
\begin{cases}
\Delta_{p}h=\lambda h^{p-1} & \text{on } M\setminus \bar B_{R}(o)\\
h=1 & \text{on }\partial B_{R}(o)\\
h \geq 0 & M \setminus B_{R}(o).
\end{cases}
\]
Then, $h(x)=h(r(x))$ is a radial function.
\end{theorem}
\begin{proof}
 Take the exhaustion $\{\O_{n} := B_{R+n}(o)\} \nearrow \mm^{m}_{\s}$ and let $\{h_{n}\}$ be the corresponding sequence of solutions of \eqref{h_n positive} which, as $n \to +\infty$, converges to the minimal solution $h$. It is enough to show that each $h_{n}$ is radial. To this end, recall from the proof of Theorem \ref{th-basic} (a) that $h_{n}$ is energy minimizing in $\S_{n}$. Therefore, if we are able to prove that the radialization $\bar h_{n}$ of $h_{n}$ belongs to $\S_{n}$ and is energy non-increasing, we must conclude that $\bar h_{n} = h_{n}$, as desired.

 Note that the radialization of $h_{n}$ is defined by
 \[
 \bar h_{n}(r) =\int_{SO(m-1)} h(r,A\cdot \theta) \, \mathrm{d}\mu(A),
 \]
where $\mathrm{d}\mu$ is the normalized Haar measure on $SO(m-1)$. In particular, $0 \leq \bar h_{n} \leq 1$ and $\bar h_{n} = 1$ on $\partial B_{R}(o)$, $\bar h_{n}=0$ on $\partial B_{R+n}(o)$. Moreover, differentiating under the integral sign, we obtain
 \[
| \nabla  h_{n}(r) | \leq \int_{SO(m-1)} |\nabla h |(r,A \cdot \theta) \, \mathrm{d}\mu(A)
 \]
 Since $p \geq 1$, by Jensen's inequality,
 \[
 \bar h_{n}^{p}(r) \leq \int_{SO(m-1)} h^{p}(r,A\cdot \theta) \, \mathrm{d}\mu(A),\quad
 | \nabla  h_{n}(r) |^{p} \leq \int_{SO(m-1)} |\nabla h |^{p}(r,A \cdot \theta) \, \mathrm{d}\mu(A).
 \]
 Finally, integrating over $B_{R+n} \setminus B_{R}$ and using Fubini we conclude
 \[
 \| \bar h \|^{p}_{L^{p}(B_{R+n} \setminus B_{R})} \leq  \| h \|^{p}_{L^{p}(B_{R+n} \setminus B_{R})},\quad
 \| \nabla \bar h \|^{p}_{L^{p}(B_{R+n} \setminus B_{R})} \leq  \| \nabla h \|^{p}_{L^{p}(B_{R+n} \setminus B_{R})},
 \]
 thus proving that $\bar h_{n} \in \S_{n}$ satisfies
 \[
 E_{p,\l}(\bar h_{n}) \leq E_{p,\l}(h_{n}).
 \]
 This completes the proof of the Lemma.
\end{proof}

%%%%%%%%%%%%%%%%%%%%%%%%%%%%%%%%%%%%%%%%%%%%%%%%%%%
%%%%%%%%%%%%%%%%%%%%%%%%%%%%%%%%%%%%%%%%%%%%%%%%%%%
\subsection{Integral Curvature Conditions}

In this section we are going to consider a  noncompact complete $m$-dimensional Riemannian manifold $(M,g)$ satisfying the lower Ricci bound
\begin{equation}\label{eq:Ric}
\ric  \geq  (m-1)\kappa g
\end{equation}
in the sense of quadratic forms, where $\kappa \leq 0$. In fact we will treat the more general case where \eqref{eq:Ric} is satisfied in a suitable (locally uniform, scale invariant) integral sense, and prove that $M$ enjoys the $p$-Feller property for every $1<p<+\infty$; see Theorem \ref{thm:p-Feller}. It is worth pointing out that this result appears to be new even in the case of the linear Feller property where $p=2$.

\subsubsection{Statement of the result} We need to introduce some notation. Let $o\in M$ be a fixed reference point and let $x$ be a generic point in $M$. We will denote by  $r(x,y)$ the Riemannian distance between two points $x,y \in M$ and when $y=o$ we will just write $r(x)$. Metric balls in $M$ of radius $R$ centered at $x$  will be denoted by $B_R(x)$. We will preferably use capital letters for constants, possibly with indices. Whenever it will be necessary we will highlight the parameter dependency, but, unless otherwise stated, constants will be independent of the radius $R$ of the ball $B_R (o)$.
\smallskip
For  $\kappa \leq 0$ we set
\[
\rho_\kappa(x)=[\min(Ric)(x) - (m-1)\kappa]_-,
\]
where $\min(Ric(x))$ denotes the smallest eigenvalue of the Ricci tensor at $x$ and $a_-=(|a|-a)/2$ is the negative part of $a$.  Note that $\rho_\kappa(x)= 0$ if and only if \eqref{eq:Ric} holds at $x$.

We also denote by $v_\kappa(r)$ the volume of the ball of radius $r$ in the $m$-dimensional space form  $\mathbb{M}_\sigma$ of constant curvature $\kappa$,  so that
\[
v_\kappa(r) = c_m \int_0^r \sigma^{m-1}(t)dt, \, \text{ with } \,
\sigma(t)=\begin{cases}
t&\text{if } \kappa =0,\\
\frac 1{\sqrt{-\kappa }}\sinh(\sqrt{-\kappa} t) &\text{if } \kappa <0.
\end{cases}
\]
For $R>0$, let also
\[
\bar k_{q,\kappa}(x,R)=
\left(
\frac 1{\vol B_R(x)}\int_{B_R(x)}\rho_\kappa (x)^q dx
\right)^{\frac 1q}\]
and
\[
\bar{k}_{q,\kappa}(R)= \sup_{x\in M} \bar{k}_{q,\kappa}(x,R).
\]
We are now in the position to state the main result of the Section.
\begin{theorem}\label{thm:p-Feller}
 Let $(M,g)$ be a complete Riemannian manifold of dimension $\dim M = m$. Given $q>m/2$ and $\kappa \leq 0$ there exists a constant $\e_{0}=\e_{0}(q,k,m)>0$ such that if
\begin{equation}\label{k-bound}
 R^{2} \bar{k}_{q,\kappa}(R) \leq \e_{0}
\end{equation}
holds for some $0 < R \leq 1$, then $M$ is $p$-Feller for every $1<p<+\infty$.
\end{theorem}
Conditions of the kind of \eqref{k-bound} was considered in several works, e.g.,  \cite{PW-GAFA, PW-TAMS, DWZ-Advances}, where the emphasis was put on the case where $\kappa=0$. Since we have not found explicit references to the results we are going to use in the case $\kappa<0$, for the convenience of the reader we will provide detailed proofs.

\subsubsection{Ingredients of the proof}
The strategy of the proof is inspired by a celebrated construction of an exhaustion function with bounded gradient and Laplacian by R. Schoen and S.T. Yau, \cite[Theorem 4.2]{SY}, recently extended to more general Ricci bounds in \cite[Theorem 2.1]{BS}. It is based on the following ingredients:
\begin{itemize}
 \item [(a)] A weighted $L^{p}$-integrability condition on the minimal solution $h$ of \eqref{p-feller}.
 \item [(b)] A lower estimate for volumes of balls of fixed radius.
\item [(c)] A local $L^{p}$ mean value inequality for $p$-subharmonic functions.
\end{itemize}
We shall see in Lemma \ref{lemma-W1p-minimal} of Section \ref{section-global-generic} that (a) holds with an exponential weight provided the constant $\l>0$ in \eqref{p-feller} is large enough.\smallskip

Concerning (b), we have the following result, which follows immediately from Proposition  \ref{Prop: LVD+VolLowerEstimate} below.

\begin{proposition}\label{prp:volume_growth}
Let $M$ be a noncompact complete manifold and let $r(x)=d(o,x)$ be the Riemannian distance from a fixed origin $o$.

Given $q>m/2$ and $\kappa \leq 0$ there exists a  constant $\epsilon_0=\epsilon_0(q, \kappa, m)>0$
such that if \eqref{k-bound} holds for some $0< R\leq 1$, then for every  $0<r\leq R/2$ there are constants $C,D>0$ depending on $q$, $\kappa,$ $m$, $o$ and $r$ such that for all $x\in M$
\begin{equation}
\label{vol-lb}
\vol B_r(x)\geq C \vol B_{r}(o) e^{-Dr(x)}.
\end{equation}
\end{proposition}

As for (c) we have the following result which is just  \cite[Proposition 1.2]{RSV-manuscripta} once it is observed that, by Proposition  \ref{Prop_LocSobolev}, the required Sobolev inequality holds for $1<p<+\infty$.

\begin{proposition}\label{thm:meanvalue}
Let $(M,g)$ be a  complete Riemannian manifold of dimension $\dim M =m$. Assume that  \eqref{k-bound} is satisfied with $\kappa\leq 0$, $q > m/2$ and some $0<R \leq 1$. Then, having fixed $1< p < +\infty$, there exist absolute constants $\a,\b>0$ and $0 < R_{1} \leq R$, depending on $\kappa$, $m$, $q$ and $p$, such that the mean value inequality
\[
u^{p}(x) \leq  \frac{\a}{(\vol B_{R_{1}}(x))^{\b}} \int_{B_{R_{1}}(x)} u ^{p}
\]
holds for every function $0\leq u \in W^{1,p}_{loc}(B_{2R_{1}}(x))\cap C^{0}(B_{2R_{1}}(x))$ satisfying $\Delta_{p} u \geq 0$.
\end{proposition}

 We are ready to give the
\begin{proof}[Proof of Theorem \ref{thm:p-Feller}]
According to Theorem \ref{th-basic}, it is enough to show that the minimal solution $h$ of \eqref{p-feller} is infinitesimal as $x \to \infty$ for some $\l > 0$ large enough.  Let us fix
$$
C=2D\b, \qquad \lambda =1+\frac{C^p}{p},
$$
where $D$ is the constant which appears in the preceding Proposition \ref{prp:volume_growth}. By Lemma \ref{lemma-W1p-minimal} of Section \ref{section-global-generic} we know that there exists a constant $A_{1}>0$ such that
$$
\int_{M \setminus B_{R/2}(o)}  \re^{C r(y)} h ^p \leq A_{1}.
$$
Next, let  $x \in M \setminus \bar{B}_{3R/2}(o)$ and $y \in B_{R/2} (x) \subset  \left( M\setminus \bar{B}_{R/2}(o)\right)$. Since
$$
r(x) - R/2\leq  r(y)
$$
then
\begin{align*}
e^{C(r(x)-1)}\int_{B_{R/2}(x)} h^p (y) &\leq \int_{B_{R/2}(x)}  e^{C r(y)} h^p (y) \\
&\leq \int_{M \setminus B_{R/2}(o)}  e^{C r(y)} h ^p(y) \\
 &\leq A_1 ,
\end{align*}
and
\begin{align*}
\int_{B_{R/2}(x)} h^p (y) &\leq    A_2e^{-C r(x)}, \qquad \mbox{with } A_2= A_1\re^{C}.
\end{align*}
Therefore, we can use the mean value inequality of Proposition \ref{thm:meanvalue} to deduce
$$
\a \textnormal{Vol}(B_{R/2} (x))^{\b} h^p(x) \leq  A_2\re^{-C r(x)},
$$
namely
\begin{equation*}\label{eq:pr1}
h^p (x) \leq A_3 \textnormal{Vol}(B_{R/2} (x)) ^{-\b} \re^{-C r(x)}.
\end{equation*}
By Proposition \ref{prp:volume_growth} and the definition of the constant $C$, we conclude that
\begin{equation*}
0<h(x) \leq A_3  \re^{(D\b-C) r(x)}= A_{3}\re^{-D\b r(x)}, \qquad \textnormal{on } M \setminus \bar{B}_{3R/2}(o).
\end{equation*}
Thus, $h(x) \to 0$ (exponentially) as $x \to \infty$.
\end{proof}

\subsubsection{Lower volume estimate and Sobolev inequality}\label{section-Appendix}
In this section we provide rather detailed proofs of the results used in the proof of Theorem \ref{thm:p-Feller}, most notably the lower estimate for the volume of $B_1(x)$ of Proposition
\ref{Prop: LVD+VolLowerEstimate} and the validity of the Sobolev inequality under integral Ricci curvature bounds contained in Proposition \ref{Prop_LocSobolev}.

We first prove that if $R\bar k_{q,\kappa}(R) $ is sufficiently small then  the measure of $M$ is uniformly locally  doubling.

\begin{lemma}
\label{Lemma LVD}
Given  $q>m/2$ and $\kappa \leq 0$, there exists a constant $\epsilon_0=\epsilon_0(q, \kappa, m)>0$ such that if
\begin{equation}
\label{k smallness condition}
R^2\bar k_{q,\kappa}(R)<\epsilon_0
\end{equation}
for some $0<R\leq 1$ then, for very $0<r_1<r_2\leq  R$ and $x\in M$, we have
\begin{equation}
\label{LVD}
\left(\vol B_{r_1}(x)\over \vol B_{r_2}(x)\right)^{\frac 1{2q}}\geq \frac 12
\left(v_\kappa(r_1)\over v_\kappa(r_2)\right)^{\frac 1{2q}}
\end{equation}
\end{lemma}
\begin{proof}
The proof follows the lines of that of Theorem 2.1 in \cite{PW-TAMS}.

Combining Lemmas 2.1 and 2.2 in \cite{PW-GAFA}  we deduce that given $q>m/2$, there exists a constant $C=C(m,q,\kappa)$ such that for every $0<r\leq 2$
\[
\frac{d}{dr}\left(\vol(B_r(x))\over v_\kappa(r)\right)\leq C \left(\vol(B_r(x))\over v_\kappa(r)\right)^{1-\frac1{2q}}\left(\int_{B_r(x)}\!\!\!\!\rho_\kappa (x)^q\right)^{\frac 1{2q}}\!v_\kappa (r)^{-\frac 1{2q}},
\]
so that, integrating between $0<r_1<r_2\leq R\leq 2$ yields
\begin{multline}
\label{A1}
\left(\vol(B_{r_2}(x))\over v_\kappa(r_2)\right)^{\frac 1{2q}}
-\left(\vol(B_{r_1}(x))\over v_\kappa(r_1)\right)^{\frac 1{2q}}\\\leq
C \left(\int_{B_{R}(x)} \!\!\!\!\rho_\kappa (x)^q\right)^{\frac 1{2q}}\int_0^{R} \!v_\kappa (r)^{-\frac 1{2q}}dr,
\end{multline}
and rearranging
\begin{multline}
\label{A2}
\left(v_\kappa(r_1)\over v_\kappa(r_2)\right)^{\frac 1{2q}}
-\left(\vol(B_{r_1}(x))\over \vol(B_{r_2}(x))\right)^{\frac 1{2q}}
\leq
\left(v_\kappa(r_1)\over v_\kappa(r_2)\right)^{\frac 1{2q}} \\
\times\left\{C
\left(v_\kappa(r_2)\over \vol(B_{r_2}(x))\right)^{\frac 1{2q}}
\left(\int_{B_{R}(x)} \!\!\!\!\rho_\kappa (x)^q\right)^{\frac 1{2q}}
\int_0^{R} \!v_\kappa (r)^{-\frac 1{2q}}dr
\right\}.
\end{multline}
Now we proceed as in \cite{PW-TAMS} Theorem 2.1, and use \eqref{A1} with $r_2$ and $R$ in place of $r_1$ and $r_2$, respectively,  to estimate
\begin{multline*}
\left(v_\kappa(r_2) \over \vol(B_{r_2}(x)) \right)^{\frac 1{2q}}\leq
\left( v_\kappa(R)\over \vol(B_{R}(x)) \right)^{\frac 1{2q}} \\ \times
\left\{
1-
C
\left(
\frac 1 {\vol B_R(x)} \int_{B_{R}(x)} \rho_\kappa (x)^q
\right)^{\frac 1{2q}}
\int_0^{R}
\left(v_\kappa(R)\over v_\kappa (r)\right)^{\frac 1{2q}}dr
\right\}^{-1},
\end{multline*}
provided the quantity at the denominator is positive. We deduce that the quantity in braces on the right hand side of \eqref{A2} is bounded above by
\[
f(x,R)=\frac
{\displaystyle{
C\bar k_{q,\kappa}(x,R)^{\frac 12}
\int_0^R
\left(v_\kappa(R)\over v_\kappa (r)\right)
^{\frac 1{2q}}dr
}}
{\displaystyle{1- C\bar k_{q,\kappa}(x,R)^{\frac 12}
\int_0^R
\left(v_\kappa(R)\over v_\kappa (r)\right)
^{\frac 1{2q}}dr.
}}
\]
Noting that the function $\phi(t)=t/(1-t)$ is strictly increasing on  $[0,1)$ and $\phi(1/3)=1/2$, we  conclude that if
\[
R\bar{k}_{q,\kappa}(R)^{\frac 12}\leq \frac{C}{\displaystyle{3\sup_{R\in(0,1]}
R^{-1}\int_0^R
\left(v_\kappa(R)\over v_\kappa (r)\right)
^{\frac 1{2q}}dr}}=\epsilon_0^{\frac 12},
\]
then $f(x,R)\leq 1/2$ and inserting into \eqref{A2} and rearranging we obtain \eqref{LVD}.
\end{proof}

As an immediate consequence we have the following

\begin{corollary}
\label{Cor LVD}
Given $q>m/2$ and $\kappa\leq 0$, assume that  $R\bar{k}_{q,\kappa}(R)\leq \epsilon_0$ for some $0< R\leq 1 $, where $\epsilon_0$ is the constant in Lemma \ref{Lemma LVD}. Then there exists a constant $C(q,\kappa, m)>0$ such that, for every $x\in M$ and $0<r\leq R/2$,
\begin{equation}
\label{LVD2}
\vol B_{2r}(x)\leq C \vol B_r(x),
\end{equation}
that is Riemannian measure on $M$ is uniformly locally doubling.
\end{corollary}
Indeed, using $r$ and $2r$ in \eqref{LVD} shows that \eqref{LVD2} holds with constant
 \[
C(q,\kappa, m)= 2^{2q}\sup_{(0,1]}
\left(v_\kappa(t)\over v_\kappa(2t)\right).
\]

Applying Lemma 5.2.7 in \cite{Sa-Aspects} we obtain the following result which immediately implies Proposition \ref{prp:volume_growth}

\begin{proposition}
\label{Prop: LVD+VolLowerEstimate}
Given $q>m/2$ and $\kappa\leq 0$, assume that  $R\bar{k}_{q,\kappa}(R)\leq \epsilon_0$ for some $0< R\leq 1 $, where $\epsilon_0$ is the constant in Lemma \ref{Lemma LVD}. Then there exist  constants $D_0$ and $D_1$ depending only on $q$,$\kappa$ and $m$ such that
 \begin{itemize}
 \item[1.] for every $x,y\in M$ and $0<r\leq R/2$,
 \[
 \vol B_r(y)\leq e^{D_0d(x,y)/r}\vol B_{r}(x);
 \]
 \item[2.] for every $x\in M$ and every $r\geq R$
 \[
 \vol B_r(x)\leq D_1^{r/R} \vol B_R(x).
 \]
 \end{itemize}
\end{proposition}

We also have the following result, which extends to the situation $\kappa<0$ the result of \cite{PW-TAMS}, Section 2.3 to the effect that if $R^2 \bar{k}_{q,\kappa}(R)\leq \epsilon_0$, where  $\epsilon_0$ is the constant in Lemma \ref{Lemma LVD}, then $ \bar{k}_{q,\kappa}(r)$ is bounded above for $0<r\leq R$ and tends to zero as $r\to 0.$ More precisely,
\begin{corollary}
\label{cor-small-r}
Given $q>m/2$ and $\kappa\leq 0$, assume that  $R\bar{k}_{q,\kappa}(R)\leq \epsilon_0$ for some $0< R\leq 1 $, where $\epsilon_0$ is the constant in Lemma \ref{Lemma LVD}. Then for every $0<r\leq R$ we have
\begin{equation}
\label{k-estimate}
r^2\bar{k}_{q,\kappa} (r)\leq
4
\left(
\frac{r^{2q}}{R^{2q}}\frac{v_\kappa(R)}{v_\kappa(r)}
\right)^{\frac 1q}R^2\bar{k}_{q,\kappa}(R) \to 0 \text{ as } r\to 0.
\end{equation}
\end{corollary}

\begin{proof}
As in \cite{PW-TAMS} we write
\[
\bar{k}_{q,\kappa}(r)= \left(\frac 1 {\vol B_r(x)}\int_{B_r(x)} \rho_\kappa(x)^q \right)^{\frac 1q} \leq \left(\vol B_R(x)\over \vol B_r(x)\right)^{\frac 1q} \bar{k}_{q, \kappa}(R),
\]
and \eqref{k-estimate} follows inserting inequality \eqref{LVD} and rearranging. Since $q>m/2$ and $v_\kappa(r)\asymp r^m$ as $r\to 0$, the right hand side of \eqref{k-estimate} goes to zero  as $r\to 0$.
\end{proof}

Using the previous result we obtain

\begin{proposition}
\label{Prop_LocSobolev}  Assume that \eqref{k smallness condition} holds for some $0<R\leq 1$. Then there exist $R_1\leq R$ and a constant $C$, depending only on $q$, $ \kappa$ and   $m$,  such that, for every $0<r\leq R_1$ and  every $f\in C^\infty_c(B_r(x))$,
\begin{equation}
\label{LocSobolev}
\left(
\int_{B_r(x)}
f^{\frac n{n+1}}
\right)
^{\frac{n+1}n}
\leq C
\frac{r}{(\vol B_r(x))^{\frac 1n}}
\int_{B_r(x)}|\nabla f|
\end{equation}
\end{proposition}
\begin{proof}
By Corollary 1.5 in \cite{DWZ-Advances} (and a scaling argument) there exists a constant $\epsilon_1=\epsilon_1(q,m)$ such that if
\[
R_1^2 \bar{k}_{q,o}(R_1)=\sup_{x\in M} R_1^2
\left(
\frac 1{\vol B_{R_1}(x)}\int_{B_{R_1}(x)}[\min Ric(x)]_-^q dx
\right)^{\frac 1q}\leq \epsilon_1
\]
then \eqref{LocSobolev} holds with $C=10^{2n+4}$ for every $0<r\leq R_1$.

Since
\[
[\min (Ric)]_- \leq [\min(Ric)-(m-1)\kappa]_-+(m-1)|\kappa|,
\]
by Minkowski inequality we have
\[
\bar{k}_{p,0}(r)\leq \bar{k}_{p,\kappa}(r)+(m-1)|\kappa|.
\]
By Corollary \ref{cor-small-r} we can choose  $R_1\leq R$ is such that for every $0<r\leq R_1$
\[
r^2\bar{k}_{q,0}(r)\leq r^2\bar{k}_{q,\kappa}(r)+(m-1)|\kappa|r^2\leq \epsilon_1,
\]
and the required conclusion follows.
\end{proof}

%%%%%%%%%%%%%%%%%%%%%%%%%%%%%%%%%%%%%%%%%%%%%%%%%%%%%%%%%%%%%%%%%%%%%%%%%%%%%%

\section{Global comparisons and Sobolev regularity: generic manifolds}\label{section-global-generic}
It was observed in \cite{PS-Feller, BPS-Spectral} that, on a stochastically complete Riemannian manifold $(M,g)$, a bounded subsolution $u$ of the exterior boundary value problem
\[
\begin{cases}
\Delta w = \lambda w \,\text{ on } \, M\setminus \bar{\Omega}\\
w\geq 0, \, w=1 \,\text{ on } \, \partial \Omega
\end{cases}
\]
is always dominated by the minimal solution $h$. On the one hand, this $L^{\infty}$-comparison gives uniqueness among bounded solutions of the problem. On the other hand, by combining the stochastic completeness with the Feller property, it implies that a nonnegative bounded subsolution $u$ must decay to $0$ at infinity.
This new phenomenon  appears e.g. on  complete Riemannian manifolds with Ricci curvature  lower bounds, and has interesting applications to the analysis of semilinear elliptic equations.\smallskip

This and the next section are devoted to extending this order of ideas to the nonlinear context of the $p$-Laplacian. We need to overcome the difficulties arising from the fact that the direct proof of the $L^{\infty}$-comparison makes an essential use of the linearity of the operator. To this end, we try to follow a different path that goes through the $L^{p}$ theory and it is related to a natural global Sobolev regularity satisfied by the minimal solution. In the route we point out a new characterization of  stochastic completeness in terms of the fact that such a
global Sobolev regularity property is enjoyed by any bounded subsolution. This new approach allows us to obtain a complete extension of \cite{PS-Feller, BPS-Spectral} to the whole scale $1< p <+\infty$ on complete manifolds with rotationally symmetric Riemannian metrics.

\subsection{Comparing with $L^{p}$-subsolutions}

%%%%%%%%%%%%%%%%%%%%%%%%%%%%%%%%%%%%%%%%%%%%%%%%%%

In this section we consider a general complete Riemannian manifold $(M,g)$ and we show how to compare a given $L^{p}$-subsolution $u$ with the minimal solution $h$ of the boundary value problem
\begin{equation}\label{p-problem}
\begin{cases}
\Delta_p w = \lambda w^{p-1} \,\text{ on } \, M\setminus \bar{\Omega}\\
w\geq 0, \, w=1 \,\text{ on } \, \partial \Omega
\end{cases}
\end{equation}
where $\O \Subset M$ is a smooth domain and $\lambda>0$ is a fixed constant. See Theorem \ref{th-lp-comparison}. To this end, we use the following rather standard result.

\begin{proposition}\label{prop-w1p-comparison}
Let $D$ be a smooth, possibly unbounded domain inside the complete manifold $(M,g)$. Let also $\lambda>0$ and $1< p < +\infty$ be fixed numbers. Assume that $u,\, v\in C^0(\bar D)\cap W^{1,p}(D)$ satisfy
\begin{equation}\label{eq-comparison}
\begin{cases}
\Delta_p u-\lambda |u|^{p-2}u\geq   \Delta_p v-\lambda |v|^{p-2}v \,\text{ in }\, D
&\\
u\leq v \,\text{ on }\, \partial D.
\end{cases}
\end{equation}
Then
\[
u\leq v \, \text{ in }\, D.
\]
\end{proposition}

\begin{proof}
Assume by contradiction that $u>v$ somewhere in $D$. Then for every
$0\leq\varrho\in C^\infty_c(M)$ the function $\phi=(u-v)_+\varrho$ is  in $W^{1,p}_{c}(D)$ and, therefore, using it as a test function in the weak definition of the inequality
\[
\Delta_p u-\lambda |u|^{p-2}u\geq   \Delta_p v-\lambda |v|^{p-2}v
\]
yields
\begin{align*}
&\int_{\{u>v\}} \varrho\langle |\nabla u|^{p-2} \nabla u - |\nabla v|^{p-2}\nabla v, \nabla u-\nabla v\rangle \\
&\leq \int_{\{u>v\}}(u-v) \cdot |  |\nabla u|^{p-2} \nabla u - |\nabla v|^{p-2}\nabla v| \cdot  |\nabla \varrho|\\
& - \lambda \int_{\{u>v\}} (|u|^{p-2}u - |v|^{p-2}v)(u-v) \varrho.
\end{align*}
It is well known that, given a finite dimensional Euclidean vector space $(V,\langle \cdot , \cdot \rangle)$, for every $x,y \in V$ we have
\begin{equation}\label{vectineq}
\langle|x|^{p-2}x - |y|^{p-2}y,x-y\rangle \geq 0
\end{equation}
where the equality sign holds if and only if $x=y$. In particular, this applies to $V= \rr$. Therefore, if we choose the cut-off function $\varrho = \varrho_{R} : M \to \rr_{\geq 0}$ in such a way that
\[
\varrho = 1\, \text{ on }  B_R(o),\quad \varrho = 0\, \text{ on } M \setminus B_{2R}(o),\quad \|\nabla \varrho \|_{\infty} \leq \frac{4}{R},
\]
it follows from  H\"older and Cauchy-Schwarz inequalities  that
\begin{align*}
&\int_{\{u>v\} \cap B_{R}(o)} \langle |\nabla u|^{p-2} \nabla u - |\nabla v|^{p-2}\nabla v, \nabla u-\nabla v\rangle\\
&\leq
\frac C R (\|u\|_p + ||v||_{p})(\|\nabla u\|_{p}^\frac{p-1}{p}+\|\nabla v\|_{p}^\frac{p-1}{p}).
\end{align*}
Thus, letting $R \to +\infty$ and using the properties of $\varrho$ we get
\[
\int_{\{u>v\}} \langle |\nabla u|^{p-2} \nabla u - |\nabla v|^{p-2}\nabla v, \nabla u-\nabla v\rangle =0
\]
which in turn, according to \eqref{vectineq}, implies that $\nabla u-\nabla v =0$ a.e. in $\{u>v\}$. Therefore $u-v$ is a positive constant in each of the connected components of $\{u>v\}$. Let $C\not= \emptyset$ be such a component. Since $D$ is locally path connected and $\{ u > v\}$ is an open subset of $D$ (here we are using that $u,v$ are continuous), $C$ is open in $D$. On the other hand, we have just deduced that $C$ is closed in $D$. It follows that $C=D$ and, therefore, $u-v \equiv c>0$ on $\bar D$. This contradicts the boundary condition satisfied by $u$ and $v$.
\end{proof}

\begin{remark}\label{rmk-w1p-comparison}
When $D$ is relatively compact, the comparison result holds true without assuming that $M$ is geodesically complete and the regularity of $u$ and $v$ can be relaxed to $W^{1,p}_{loc}(D) \cap C^{0}(\bar D)$. Actually, if the inequalities in \eqref{eq-comparison} are replaced by the corresponding equalities, the above arguments yield the following uniqueness result without any continuity assumption on the solutions.\smallskip

\noindent {\it Let $D$ be a smooth, relatively compact domain inside $(M,g)$. Let also $\lambda>0$ and $1< p < +\infty$ be fixed numbers. Assume that $u,\, v\in W^{1,p}(D)$ satisfy
\begin{equation*}
\begin{cases}
\Delta_p u-\lambda |u|^{p-2}u =   \Delta_p v-\lambda |v|^{p-2}v \,\text{ in }\, D
&\\
u =  v \,\text{ on }\, \partial D,
\end{cases}
\end{equation*}
where the boundary data are understood in the trace sense. Then
\[
u =  v \, \text{ a.e. in }\, D.
\]
}
\smallskip

\noindent Indeed, according to a characterization of the zero-boundary trace condition we still have $u-v \in W^{1,p}_{0}(D)$. Therefore, using the test function $\phi = (u - v)_{+} \in W^{1,p}_{0}(D)$ in the weak definition of the differential equality we get $\nabla u = \nabla v$ on $\{ u \geq v\}$. Similarly, by reversing the role of $u$ and $v$, $\nabla u = \nabla v$ on $\{v \geq u\}$. Since $D$ is connected, it follows that $u=v$ a.e. in $D$.
\end{remark}

We stress that the above comparison does not require any geometric assumption on the manifold.
It depends entirely on the Sobolev class of the functions involved. Let us observe that the minimal solution of \eqref{p-problem} belongs automatically to the required Sobolev class. Actually, more is true. First, we fix a notation. Let $D \subseteq M$ be a domain and  $0\leq \varrho \in L^{1}_{loc}(D)$ be a fixed weight. We set
\[
W^{1,p}_{\varrho}(D) = \{ u \in L^{p}(D,\varrho \,dv) : \nabla u \in L^{p}(D,\varrho \,dv) \},
\]
where $dv$ denotes the Riemannian measure.
\begin{lemma}\label{lemma-W1p-minimal}
Let $\O \Subset M$ be a smooth domain. Let $h \in C^{1,\alpha}_{loc}(M\setminus {\Omega})$ be the minimal solution of problem \eqref{p-problem}. Then
\[
h\in W^{1,p}_{\varrho}(M\setminus \bar{\Omega}) \cap L^{\infty}(M\setminus {\Omega}),
\]
where $\varrho(x) = \re^{Cr(x)}$,  $r(x)$ is the distance from a fixed origin $o \in M$ and $0< C < (\l p)^{\frac1 p} $
\end{lemma}

\begin{proof}
Choose a sequence $\{R_{n}\} \nearrow +\infty$ such that $\O \Subset B_{R_{0}}(o)$. Next, recall that $h$ is obtained as the limit function of the sequence $\{h_n\}$, where for every fixed $n$,
$$
\begin{cases}
\Delta_{p} h_{n} = \lambda h_{n}^{p-1} & \mbox{on } \bar{B}_{R_{n}}(o)\setminus \O,\\
h=1 & \mbox{on } \partial \O,\\
h= 0 & \mbox{on }  B_{R_{n}}(o).
\end{cases}
$$
Then, denoting  by ${\nu}$ the outward  unit normal to  $\partial (\bar{B}_{R_{n}}(o)\setminus \O)$, we have
\begin{align*}\label{eq:integ_parts}
\int_{B_{R_{n}}(o) \setminus \O} \re^{C r}h_{n} \Delta_p h_{n} &=\int_{\partial \O\cup\partial B_{R_{n}}(o) }\textrm{\re}^{Cr} h_n |\nabla h_n|^{p-2}\frac{\partial h_n }{\partial {\nu}} \\
 &- \int_{B_{R_{n}}(o) \setminus \O} \langle \nabla ( \re^{Cr }h_n ),|\nabla h_n |^{p-2} \nabla h_n  \rangle_g \\
&=\int_{\partial \O}  \re^{Cr}|\nabla h_n |^{p-2} \frac{\partial h_n }{\partial{\nu}} \\
&- \int_{B_{R_{n}}(o) \setminus \O}C  \re^{C r} h_n  |\nabla h_n |^{p-2}\langle \nabla r, \nabla h_n  \rangle_g \nonumber\\
&  - \int_{B_{R_{n}}(o) \setminus \O}  \re^{C r} |\nabla h_n |^p.
\end{align*}
Since $0 \leq h_n \leq 1$ for every $n$, by Schauder estimates up to the boundary
$$
\left\vert \int_{\partial \O}  \re^{Cr}|\nabla h_n |^{p-2} \frac{\partial h_n }{\partial{\nu}} \right\vert
\leq \int_{\partial \O} \re^{Cr}|\nabla h_{n}|^{p-1}  \leq A_1,
$$
where the constant $A_1>0$ is independent of $n$. Then, from the above inequality and using that $\Delta_p h_n  = \lambda h_n ^{p-1}$, we get		
\begin{align*}
\int_{B_{R_{n}}(o) \setminus \O}  \lambda\re^{C r}h_n ^p &\leq A_1   - \int_{B_{R_{n}}(o) \setminus \O} C  \re^{C r} h_n   |\nabla h_n |^{p-2}\langle \nabla r , \nabla h_n  \rangle_g \nonumber\\
&  - \int_{B_{R_{n}}(o) \setminus \O}  \re^{C r} |\nabla h_n |^p, \nonumber\\
&\leq A_1 +  \int_{B_{R_{n}}(o) \setminus \O}  C  \re^{C r} h_n  | \nabla h_n  |^{p-1}\\
& - \int_{B_{R}(o) \setminus \O}  \re^{C r} |\nabla h_n |^p\\
& \leq A_1 + \int_{B_{R_{n}}(o) \setminus \O} \left( \frac{C^p}{p}h_n ^p + \frac{|\nabla h_n |^p}{\frac{p}{p-1}}  \right)\re^{C r} \\
& - \int_{B_{R_{n}}(o) \setminus \O}  \re^{C r} |\nabla h_n |^p,
\end{align*}
whence, rearranging,
$$
\left(\l - \frac{C^{p}}{p}\right)\int_{B_{R_{n}}(o) \setminus \O}  \re^{C r} h_n  ^p
+ \frac{1}{p} \int_{B_{R_{n}}(o) \setminus \O} \re^{Cr} |\nabla h_{n}|^{p}
\leq A_1.
$$
Recalling that $h_{n} \to h$ where the convergence is  $C^{1}$ on compact sets, the desired conclusion follows by letting $R \to +\infty$ and by using Fatou's Lemma.
\end{proof}

Another interesting fact related to the Sobolev regularity required by Proposition \ref{prop-w1p-comparison} is that, if we limit ourselves to positive  solutions of $\Delta_p u \geq \lambda |u|^{p-2}u$ on exterior domains, then the $L^{p}$-integrability of $\nabla u$ follows from the $L^{p}$-integrability of the function. This is a simple consequence of the Caccioppoli inequality as we are going to show.

In what follows, given a set $F \subseteq M$ and a real number $\d>0$ we denote by $F_{\d}$  the $\d$-neighborhood of $F$.

\begin{lemma}\label{lemma-w1p-sub}
Let $\O \Subset M$ be a  domain.  Let $u\in W^{1,p}_{loc}(M\setminus \bar \Omega)$  be a non-negative solution of the differential inequality
\[
\Delta_p u \geq \lambda u^{p-1}\, \text{ on } \, M\setminus \bar \Omega,
\]
for some $\lambda >0$.
If $u\in L^p(M\setminus   \Omega)$  then, for any $\e>0$,
\[
| \nabla u | \in L^p(M\setminus {\Omega}_{\e}).
\]
\end{lemma}

\begin{proof}
By definition, for every $0\leq\varphi\in W_0^{1,p}(M\setminus \bar {\Omega})$,
\[
-\int_{M\setminus \bar\Omega} |\nabla u|^{p-2}\langle \nabla u, \nabla \varphi\rangle \geq
\lambda \int_{M\setminus \bar\Omega} u^{p-1}\varphi.
\]
For every $R \gg 1$, let $\varrho = \varrho_{R}\in C^\infty_c(M)$ be a cutoff function satisfying
\[
i)\, \varrho = 0 \text{ outside } B_{R+1}\setminus \O_{\e/2};\quad ii)\, \varrho = 1 \, \text{in }B_{R} \setminus \O_{\e}, \quad iii)\, \| \nabla \varrho \|_{\infty} =  \CO(R^{-1}).
\]
Using $\varphi=\varrho^p u$ as a test function in the above inequality yields
\[
\int \varrho^p |\nabla u|^p \leq -\lambda \int \varrho^p u^p +p\int u \varrho^{p-1} |\nabla u|^{p-1}|\nabla \varrho|.
 \]
By Young's inequality, the integrand in the second integral on the RHS is estimated by
 \[
 \frac 1 {\d ^p} u^p + (p-1)\d^{p/(p-1)} \varrho^p |\nabla u|^p,
 \]
where $\d >0 $ is such that $(p-1)\d^{p/(p-1)}<1/2$. Inserting into the above inequality yields
 \[
 (1-(p-1)\d^{p/(p-1)})\int \varrho^p |\nabla u|^p \leq  -\lambda \int \varrho^p u^p
 + \d^{-p} \int|\nabla \varrho|^p u^p,
 \]
whence, using the properties of $\varrho$ and letting $R\to \infty$, we conclude that
\[
\int_{M\setminus \O_{\e}}|\nabla u|^p \leq C \e^{-p}\int_{M\setminus \O_{\e/2} }u^p<+\infty.
\]
\end{proof}

Thus, if we combine Proposition \ref{prop-w1p-comparison} with Lemma \ref{lemma-W1p-minimal} and Lemma \ref{lemma-w1p-sub} we conclude the validity of the following result.

\begin{theorem}\label{th-lp-comparison}
 Let $\O \Subset M$ be a smooth domain inside the complete manifold $(M,g)$. Let  $h \in C^{1,\alpha}_{loc}(M \setminus  \O)$ be the minimal solution of \eqref{p-problem} and let $u\in W^{1,p}_{loc}(M \setminus \bar \O)\cap C^{0}(M \setminus  \O)$ be a nonnegative function satisfying
 \[
 \Delta_{p}u \geq \l u ^{p-1}\, \text{on }M \setminus \bar \O.
 \]
If $u \in L^{p}(M \setminus \bar \O)$ then, for any $0 < \e  \ll 1$, it holds
\[
u \leq c \cdot h\, \text{ on }M \setminus \O_{\e},
\]
where $c  =\max_{\partial \O_{\e}} (u +\e)/h$.
In particular:
\begin{itemize}
 \item [i)] $u$ is bounded;
\item [ii)] $u \in W^{1,p}(M\setminus \O_{\e})$;
\item [iii)] $u(x) \to 0$ as $x \to \infty$ provided $(M,g)$ is $p$-Feller.
\end{itemize}
\end{theorem}
 In the absence of any further geometric/stochastic assumption on $M$, Theorem \ref{th-lp-comparison} may represent the natural $L^{p}$ analogue of the $L^{\infty}$ comparison observed in \cite{PS-Feller, BPS-Spectral} which is valid for $p=2$ on a stochastically complete manifold; see Lemma \ref{lemma-Linfty-comparison} in Section \ref{section-sobolevreg} .
A natural question arises:
\begin{problem}\label{problem-lp}
 Is it possible to prove directly that, under suitable conditions on $M$, such as a Ricci lower bound or the $p$-analogue of the stochastic completeness (see Section \ref{section-sobolevreg}), any $L^{\infty}$-subsolution of the boundary value problem \eqref{p-problem} is in the $L^{p}$ class?
\end{problem}

Thanks to Theorem \ref{th-lp-comparison}, a positive answer to Problem \ref{problem-lp} would give a proof of the following ``Ansatz'', which  extends to any $1<p<+\infty$ the above mentioned $L^{\infty}$ comparison of \cite{PS-Feller, BPS-Spectral}. \smallskip

\begin{ansatz}\label{ansatz}
 Let $(M,g)$ be a ``suitable'' complete Riemannian manifold (e.g., whose Ricci curvature is bounded from below or $p$-stochastically complete in the sense of Definition \ref{def-pstochcompl} below) and let $\O \Subset M$ be a smooth domain. If $u \in C^{0}(M \setminus \O) \cap W^{1,p}_{loc}(M \setminus \bar \O)$ is a nonnegative and bounded function satisfying
 \[
 \Delta_{p} u \geq \l u^{p-1}\, \text{on } M \setminus \bar \O
 \]
 for some $\l >0$ then, there exist a smooth domain $\Omega_1$ and a constant $c = c(\O_{1})>0$  such that $\O \Subset \O_{1} \Subset M$  and
 \[
 u  \leq c \cdot h \quad M \setminus \O_{1}
 \]
 where $h \in C^{1}(M \setminus \O)$ is the minimal solution of \eqref{p-problem}.
 More generally, this comparison property holds when $h$ is replaced by positive supersolution.
\end{ansatz}

We shall see in Section \ref{section-rotational} that the answer to Problem \ref{problem-lp}  is yes if the underlying manifold is rotationally symmetric (but, in fact, the symmetry condition plays no role if $p=2$ as a consequence of Proposition \ref{th-characterization} of Section \ref{section-sobolevreg}).\smallskip

%%%%%%%%%%%%%%%%%%%%%%%%%%%%%%%%%%%%%%%%%%%%%%%%%%%
\subsection{Global Sobolev regularity of  subsolutions}\label{section-sobolevreg}

Apparently, the first connection between the Sobolev space $W^{1,2}(M)$ and the stochastic completeness of the underlying manifold $(M,g)$ was pointed out by A. Grigor'yan and J. Masamune in the recent paper \cite{GM-JMPA}. More precisely, in \cite[Theorem 1.7]{GM-JMPA} it is proved that if $M$ is stochastically complete then $C^{\infty}_{c}(M)$ is dense in $W^{1,2}(M)$. This result, in turn, has interesting implications to the essential self-adjointness of the Laplacian in non-complete settings, and to the validity of a global Stokes theorem for weakly regular functions. In this section we shall focus on geodesically complete manifolds and present a new characterization of  stochastic completeness which still involves the space $W^{1,2}(M)$.
Actually, we are going to extend the investigation to the more general setting of the $p$-Laplacian, $1<p<+\infty$. Following \cite{PRS-revista, MV-tams}, we record the following:

\begin{definition}\label{def-pstochcompl}
A Riemannian manifold $(M,g)$ is said to be $p$-stochastically complete if either of the following equivalent conditions is satisfied.
\begin{itemize}
\item [a)] (Liouville property) For some (hence any) $\lambda >0$, the only (weak) solution $u\in C^{0}(M)\cap W^{1,p}_{loc}(M)$ of the problem
\[
\begin{cases}
\Delta_p u \geq \lambda u^{p-1}, \,\, \text{on } M\\
0\leq u \leq \sup_M u  <+\infty
\end{cases}
\]
is the constant function $u \equiv 0$.
\item [b)] (Maximum principle at infinity) For every function $u\in C^{0}(M)\cap W^{1,p}_{loc}(M)$ satisfying $\sup_{M} u = u^{\ast} <+\infty$ and for every $\gamma< u^{\ast}$ it holds
\[
\inf_{\{u>\gamma\}} \Delta_p u  \leq 0.
\]
Here, the infimum has to be understood in the following sense: for any $\e>0$ there exists a test function $0 \leq \vp \in C^{\infty}_{c}(\{ u > \gamma\})$ such that $-\int_{M} \langle |\nabla u|^{p-2}\nabla u, \nabla \vp\rangle < \e$.
\end{itemize}
\end{definition}
It is the main contribution of \cite{PRS-stoch} that, when $p=2$, the above conditions (in fact condition b)) are equivalent to the the usual notion of stochastic completeness, i.e., that  Brownian motion on $M$ does not explode in finite time. On the geometric side, as a very special case of \cite[Corollary 2.4]{PRS-revista} we have that a complete manifold $(M,g)$ satisfying $\ric \geq (m-1) \kappa$, $\kappa\leq 0$ is $p$-stochastically complete, for every $1<p<+\infty$.\smallskip

For the sake of convenience, we also set the following
\begin{definition}
Let $1<p<+\infty$. We say that the Riemannian manifold $(M,g)$ satisfies the global Sobolev regularity property $(\SR_{p})$ if the following holds.\smallskip

\noindent $(\SR_{p})$ Let $\O_{0} \Subset M$ be a given domain, $\l>0$ a fixed real number and
\[
u \in W^{1,p}_{\mathrm{loc}}(M\setminus \bar \O_{0}) \cap C^{0}(M \setminus \O_{0})
\]
a nonnegative and bounded function satisfying the  differential inequality
\[
\Delta_{p} u \geq \l u^{p-1}, \,M \setminus \bar \O_{0}.
\]
Then there exists a smooth domain $\O_{0} \subseteq  \O_{1} \Subset M$ such that
\[
u \in W^{1,p}(M \setminus   \bar \O_{1}).
\]
\end{definition}

The main result of the present section is the next global Sobolev regularity result.

\begin{proposition}\label{th-characterization}
 Let $(M,g)$ be a geodesically complete manifold and let $1<p<+\infty$. Then:
\begin{itemize}
\item [(a)] If $p=2$ and $M$ is stochastically complete, then $M$ enjoys  $(\SR_{2})$.
\item [(b)] If $M$ satisfies property $(\SR_{p})$ then $M$ is $p$-stochastically complete.
\end{itemize}
In particular, for $p=2$, the global Sobolev regularity property $(\SR_{2})$ is equivalent to the stochastic completeness of the complete Riemannian manifold $(M,g)$.
\end{proposition}

\begin{remark}
We shall see in Section \ref{section-rotational} that, in the special case of model manifolds, the equivalence between $(\SR_{p})$ and  the $p$-stochastic completeness of the underlying manifold holds on the whole $L^{p}$ scale, $1<p<+\infty$.
\end{remark}
\begin{remark}
 As it will be apparent from the proof, in the case where $p=2$, we are able to prove that property $(\SR_{2})$ holds true with $\O_{0} = \O_{1}$ (and without the assumption that $M$ is geodesically complete). In particular $|\nabla u|$ is in $L^{2}$ even in a neighborhood of $\partial \O_{0}$. In the attempt of extending the analysis to all $1< p < +\infty$, perhaps via different tools such as global $L^{p}$ comparisons, we are willing to weaken the conclusion by assuming that $\nabla u$ is in $L^{p}$ at infinity. As the previous remark shows, this is a good strategy  in the setting of model manifolds.
\end{remark}

We need the next simple $L^{\infty}$ comparison that was systematically used in [\cite{PS-Feller, BPS-Spectral}]. In these papers the comparison property was not stated as a formal result so it is difficult to quote it appropriately.  We take the occasion to state it here formally in a slightly more general version.

\begin{lemma}[\cite{PS-Feller, BPS-Spectral}]\label{lemma-Linfty-comparison}
Let $(M,g)$ be a stochastically complete manifold and let $D \subset M$ be a (possibly unbounded) domain. Let $u,v \in C^{0}(\bar D) \cap W^{1,2}_{loc}(D)$ be  solutions of
\[
\begin{cases}
 \Delta u - \lambda u \geq  \Delta v - \lambda v, & D \\
 u \leq v, & \partial D \\
\end{cases}
\]
for some $\l >0$. If
\[
\sup_{D} (u-v) <+\infty
\]
then
 \[
 u  \leq v \quad D.
 \]
\end{lemma}

\begin{proof}
 Take any $\e>0$ and consider the function $w_{\e} = \max(u - v -\e,0)$. Note that $w_{\e} \geq 0$ is a bounded solution of $\Delta w_{\e} \geq \l w_{\e}$ on $D$. Since $w_{\e}$ is supported inside $D$ we can extend it to a bounded solution of the same inequality on all of $M$. By stochastic completeness, $w_{\e} \equiv 0$ proving that $u \leq v + \e$. To conclude, we  let $\e \searrow 0$.
\end{proof}

We are in the position to give the

\begin{proof}[Proof of Proposition \ref{th-characterization}]
 (a)  Assume that $M$ is stochastically complete and let $\O_{0},\l$ and $u$ be as in $(\SR_{2})$. Let $0 \leq h \leq 1$ be the minimal positive solution of the problem
 \[
 \begin{cases}
 \Delta h  = \l h, \,M \setminus \bar \O_{0} \\
 h = 1,\, \partial \O_{0}
\end{cases}
 \]
 By Lemma \ref{lemma-Linfty-comparison} $u \leq c h$ on $M \setminus \O_{0}$, for some constant $c>0$. The desired integrability properties of $u$ and $|\nabla u|$ now  follow from Lemma \ref{lemma-W1p-minimal} and Lemma \ref{lemma-w1p-sub}. This completes the proof of part (a).\\

(b) Conversely, assume that $(\SR_{p})$ HOLDS and suppose by contradiction that $M$ is not $p$-stochastically complete. Then, there exists a function $u \in C^{0}(M) \cap W^{1,p}_{loc}(M)$ satisfying
 \[
 \begin{cases}
 \Delta_{p} u  \geq \l u^{p-1}, \,M \\
0 \leq u \leq \sup_{M} u < +\infty,
\end{cases}
 \]
By $(\SR_{p})$ we know that $u \in W^{1,p}(M)$ (but $L^{p}$ suffices). Since the volume of $M$ is necessarily infinite, the $p$-version of the classical Liouville theorem by Yau, see e.g. \cite[Theorem 2.30]{PS-Ensaios}, yields $u \equiv 0$. The contradiction completes the proof. Alternatively, one can use the Stokes theorem by Gaffney applied to the vector field $X = u |\nabla u|^{p-2} \nabla u$. In fact, note that $|X| \leq \frac{1}{p}|u|^{p}+ \frac{p-1}{p}|\nabla u|^{p} \in L^{1}(M)$ and $\div X = |\nabla u|^{p}+ \l u^{p} \in L^{1}(M)$.
\end{proof}

\section{Global comparisons and Sobolev regularity: model manifolds}\label{section-rotational}
In this Section we give a proof of  (more general versions of) Theorem \ref{th-main1} and Theorem \ref{th-main2} stated in the Introduction. According to the discussions in Section \ref{section-global-generic}, everything boils down to the following key result.
\begin{theorem}\label{th-Lpmodel}
 Let $\mm^{m}_{\s}$ be a $p$-stochastically complete model manifold, $1<p<+\infty$. Let $u \in W^{1,p}_{loc}(\mm^{m}_{\s} \setminus \bar B_{R}(0)) \cap C^{0}(\mm^{m}_{\s} \setminus B_{R}(0))$ be a nonnegative and bounded solution of
 \[
 \Delta_{p} u \geq \l u^{p-1},\, \mm^{m}_{\s} \setminus \bar B_{R}(0)
 \]
 for some $R>0$ and $\l>0$. Then
 \[
 u \in L^{q}(\mm^{m}_{\s} \setminus \bar B_{R}(0)),
 \]
 for every $q \in [p-1 , +\infty]$.
\end{theorem}
The proof of Theorem \ref{th-Lpmodel} is obtained in two steps: first, making use of a volume growth characterization of the $p$-stochastic completeness in model manifolds, we study the case of radial solutions of the exterior problem and show that they are in $L^{q}$. Next, we compare the general subsolution with the radial solution.\smallskip

According to this strategy we start by recalling the following fact from \cite{PRS-revista}.
\begin{lemma}
Let $1<p<+\infty$. A necessary and sufficient condition for the complete model manifold $\mm^{m}_{\s}$ to be $p$-stochastically complete is that
\begin{equation}\label{pstoch-models}
\left(\frac{\int_{0}^{r}\s^{m-1}(t)dt}{\s^{m-1}(r)}\right)^{\frac{1}{p-1}} \not\in L^{1}(+\infty)
\end{equation}
\end{lemma}
\begin{remark}
 As in the linear case, every $p$-parabolic manifold is $p$-stochastically complete. In particular, \eqref{pstoch-models} is implied by the validity of \eqref{model-parab}.
\end{remark}
The following result greatly extends and strengthen \cite[Lemma 4.2]{PS-Feller}.
\begin{lemma}\label{lemma-radialsolution}
 Let $\mm^{m}_{\s}$ be a $p$-stochastically complete model, $1<p<+\infty$. Let $u(x) = u(r(x))$ be a nonnegative and bounded  $C^{1}$ function on  $\mm^{m}_{\s} \setminus  B_{R}(0)$ that satisfies, in the weak sense, the equation
\begin{equation}
\label{p-eq-radial}
  \Delta_{p} u = \l u^{p-1}, \text{ on } \mm^{m}_{\s}\setminus \bar B_{R}(0)
\end{equation}
for some $\l >0$. Then:
\begin{itemize}
 \item [(a)] $u(r)$ is  nonincreasing  on $[R, +\infty)$.
 \item [(b)] $u \in L^{q}(\mm^{m}_{\s}\setminus  B_{R}(0))$, for every $q \in [p-1,+\infty]$.
\end{itemize}
\end{lemma}
\begin{remark}
An inspection of the proof shows that everything works for a bounded solution of the more general equation
\[
\Delta_{p} u = g(u), \text{ on } \mm^{m}_{\s}\setminus \bar B_{R}(0)
\]
where $g: \rr_{\geq 0} \to \rr_{\geq 0}$ is a continuous non-decreasing function satisfying $g(t)>0$ for $t>0$. In this case, conclusion (b) takes the form:\smallskip

(b') $g(u) \in L^{1}(\mm^{m}_{\s} \setminus B_{R}(0)) \cap L^{\infty}(\mm^{m}_{\s} \setminus B_{R}(0))$.
\end{remark}
\begin{proof}
We begin by noting that, by elliptic regularity, $u$ is $C^{2}$ on the open set $\{ \nabla u \not= 0\}$. Moreover, since $u$ is a radial distributional solution of \eqref{p-eq-radial} then, in particular, for every $\vp \in \mathrm{Lip}_{c}((R,+\infty))$ it holds
\[
-\int_{R}^{+\infty} |u'|^{p-2}u' \vp' \s^{m-1} = \l \int_{R}^{+\infty} u^{p-1} \vp \s^{m-1},
\]
which, on the set  $\{u' \not=0\}$, takes the pointwise form
\[
\left( \s^{m-1}|u'|^{p-2}u' \right)' = \l \s^{m-1}u^{p-1}.
\]
We are now ready to start the proof of the Lemma.\smallskip

(a) We  have to show that the open set
\[
D_{+} = \{ u' > 0\}
\]
is empty. Suppose by contradiction that there exists $r_{0} \in D_{+} \not= \emptyset$. Obviously, up to moving $r_{0}$ a little bit, we can assume that $u(r_{0})>0$. Since, on $D_{+}$,
\begin{equation}\label{odeD+}
\left( \s^{m-1}(u')^{p-1} \right)' = \l \s^{m-1}u^{p-1} \geq 0
\end{equation}
it follows that the function $\s^{m-1}(u')^{p-1}$ is  nondecreasing  on  the maximal interval $[r_{0},r_{1}) \subseteq D_{+}$. This implies that
\[
u'(r) \geq \s^{\frac{1-m}{p-1}}(r)\s^{\frac{m-1}{p-1}}(r_{0}) u'(r_{0}) >0, \text{ on } [r_{0},r_{1})
\]
and, therefore, $r_{1} = +\infty$. Integrating \eqref{odeD+}  from $r_{0}$ to any $r>r_{0}$ gives
\begin{align*}
\s^{m-1}(r)(u')^{p-1}(r) &\geq \s^{m-1}(r)(u')^{p-1}(r) - \s^{m-1}(r_{0})(u')^{p-1}(r_{0}) \\
&= \l \int_{r_{0}}^{r} u^{p-1}(t) \s^{m-1}(t)dt\\
&\geq \l u^{p-1}(r_{0}) \int_{r_{0}}^{r} \s^{m-1}(t) dt.
\end{align*}
Thus,  there exists a constant $C=C(r_{0},u,\s)>0$ such that
\[
u'(r) \geq C \left( \frac{\int_{r_{0}}^{r} \s^{m-1}(t)dt}{\s^{m-1}(r)} \right)^{\frac{1}{p-1}}.
\]
Integrating this latter on $[r_{0},+\infty)$ and recalling that, by the $p$-stochastic completeness of $\mm^{m}_{\s}$, condition \eqref{pstoch-models} is satisfied, we conclude that
\[
u(r) \to +\infty \text{ as }r \to +\infty.
\]
This contradicts the assumption that $u$ is bounded and completes the proof of (a).\smallskip

(b) Recall that, by (a), $u' \leq 0$. In particular, for every $\vp \in \mathrm{Lip}_{c}((R,+\infty))$,
\[
\int_{R}^{+\infty} (-u')^{p-1} \vp' \s^{m-1} = \l \int_{R}^{+\infty} u^{p-1} \vp \s^{m-1}
\]
For every $r > R+2$, we choose $\vp= \vp_{r}$ to be piecewise linear and  such that $0 \leq \vp \leq 1$, $\mathrm{supp}(\vp) \subset [R+1, r+1]$ and $\vp = 1$ on $[R+2,r]$. We thus obtain
\begin{align*}
\int_{R+1}^{R+2}  (-u')^{p-1}  \s^{m-1} &\geq \int_{R+1}^{R+2}  (-u')^{p-1}  \s^{m-1} - \int_{r}^{r+1} (-u')^{p-1} \s^{m-1}\\
&= \l \int_{R+1}^{r+1}u^{p-1}\vp \s^{m-1}\\
&\geq \l \int_{R+2}^{r}  u^{p-1}\s^{m-1}.
\end{align*}
Letting $r \to +\infty$ we deduce that $u^{p-1} \s^{m-1} \in L^{1}([R+2,+\infty))$ and this latter, clearly, is equivalent to  $u \in L^{p-1}(\mm^{m}_{\s}\setminus B_{R+2}(0))$. By interpolation, since $u$ is bounded, we conclude the $L^{q}$ integrability of $u$ for every $q \in [p-1,+\infty]$.
\end{proof}

We are now in the position to give the
\begin{proof}[Proof of Theorem \ref{th-Lpmodel}]
 Clearly, up to rescaling $u$, we can  assume that
 \[
 u \leq 1,\text{ on }\partial B_{R}(0).
 \]
 Let $u^{\ast} = \sup_{\mm^{m}_{\s} \setminus B_{R}(0)} u < +\infty$ and, for every $R+n> R$,  define
 \[
 M_{n} = \max_{\partial B_{R+n}(0)} u \leq u^{\ast}.
 \]
As in the proof of Theorem \ref{th-basic} (a), we consider the (energy-minimizing) solution $v_{n} \in C^{1,\a}(\bar B_{R+n}(0) \setminus  B_{R}(0))$ of the boundary value problem
 \begin{equation}\label{boundary-r}
\begin{cases}
 \Delta_{p} v_{n} = \l |v_{r}|^{p-2}v_{n},& B_{R+n}(0) \setminus \bar B_{R}(0)\\
 v_{n} = 1, & \partial B_{R}(0)\\
 v_{n} = M_{n} & \partial B_{R+n}(0).
\end{cases}
 \end{equation}
By the arguments in Lemma \ref{lemma-symmetric}, $v_{n}$ is radial and, by the comparison principle,
\[
 0 \leq u \leq v_{n} \leq \max\{ 1,M_{n} \} \leq \max\{1,u^{\ast}\},\, \text{ on }\bar B_{R+n}(0) \setminus B_{R}(0).
\]
In particular the sequence $\{ v_{n} \}$ is uniformly bounded. Therefore, using Lieberman's Schauder estimates up to the boundary, we see that a subsequence $\{ v_{n'}  \}$ converges to a radial $C^{1}$ solution $v(x) = v(r(x))$ of the problem
\[
\begin{cases}
 \Delta_{p} v = \l v^{p-1} & \mm^{m}_{\s} \setminus \bar B_{R} \\
 v =1 & \partial B_{R}(0).
\end{cases}
\]
Note that, by construction,
\[
0 \leq u \leq v \leq u^{\ast}.
\]
Therefore, an application of Lemma \ref{lemma-radialsolution} yields that $v$, and hence $u$, are in $L^{q}$ for every $q \in [p-1,+\infty]$. This completes the proof of the Theorem.
\end{proof}

As a consequence of Theorem \ref{th-Lpmodel} let us show how to obtain Theorems \ref{th-main1} and \ref{th-main2}. Recall that a complete model manifold $\mm^{m}_{\s}$ satisfying $\ric \geq - \kappa^{2}$ is both $p$-stochastically complete and $p$-Feller.

\begin{theorem}\label{th-main1a}
 Let $\mm^{m}_{\s}$ be a $p$-stochastically complete and $p$-Feller model manifold,  $1<p<+\infty$. Let $\O \Subset \mm^{m}_{\s}$ be a smooth domain and let $u \in C^{0}(\mm^{m}_{\s} \setminus \O) \cap W^{1,p}_{loc}(\mm^{m}_{\s} \setminus \bar \O)$ be a nonnegative and bounded function satisfying
\[
 \Delta_{p} u \geq \Lambda(u)\, \text{on } \mm^{m}_{\s} \setminus \bar \O
\]
where $\Lambda : [0,+\infty) \to [0,+\infty)$ is a non-decreasing function such that
\[
(i)\, \Lambda\left(  0\right)  =0,\,\,\, (ii) \, \Lambda\left(  t\right) >0\text{, }\forall t>0, \,\,\, (iii) \liminf_{t\rightarrow0+}\frac {\Lambda\left(  t\right)  }{t^{\xi}}>0,
\]
for some $0\leq\xi\leq p-1$. Then
 \[
 u(x) \to 0\, \text{as }x\to \infty.
 \]

\end{theorem}
\begin{proof}
Note that, since $0 \leq u(x) \leq u^{\ast}<+\infty$, the structural conditions on $\Lambda$ imply that
 \[
 \Delta_{p} u \geq \l  u^{p-1}\, \text{on } \mm^{m}_{\s} \setminus \bar \O
 \]
 for a suitable constant $\l >0$. Indeed, by (iii), up to choosing $0<t^{\ast}\ll 1$ we have $\Lambda(t) \geq C t^{\xi}$ on $[0,t^{\ast}]$ for some constant $C>0$. Therefore, if $u \geq t^{\ast}$, since  $\Lambda$ is non-decreasing, we get
 \[
 \Lambda(u) \geq \Lambda(t^{\ast}) = \frac{\Lambda(t^{\ast}) }{(u^{\ast})^{p-1}} (u^{\ast})^{p-1} \geq  \frac{\Lambda(t^{\ast})}{(u^{\ast})^{p-1}}{u^{p-1}}.
 \]
 On the other hand, if $u < t^{\ast} <1$,
 \[
 \Lambda(u) \geq C u^{\xi} \geq C u^{p-1}
 \]
Therefore, we can take $\l = \min(C,  \Lambda(t^{\ast})/ (u^{\ast})^{p-1})$.\smallskip

 Since $\mm^{m}_{\s}$ is $p$-stochastically complete, by Theorem \ref{th-Lpmodel} we know that $u \in L^{p}(\mm^{m}_{\s} \setminus B_{R}(0))$. Therefore, using the fact that $\mm^{m}_{\s}$ is $p$-Feller, we can apply Theorem \ref{th-lp-comparison} iii) and conclude that $u(x) \to 0$ as $x \to \infty$.
\end{proof}

\begin{theorem}\label{th-main2a}
 The model manifold  $\mm^{m}_{\s}$ is $p$-stochastically complete if and only if the global Sobolev regularity property $(\SR_{p})$ is satisfied.
\end{theorem}
\begin{proof}
By Proposition \ref{th-characterization} we have only to show that if $\mm^{m}_{\s}$ is $p$-stochastically complete then $(\SR_{p})$ holds true. Thus, fix any  $\O_{0} \Subset \mm^{m}_{\s}$, a real number $\l>0$ and a positive and bounded function  $u \in W^{1,p}_{loc}(\mm^{m}_{\s} \setminus \bar \O_{0}) \cap C^{0}(\mm^{m}_{\s} \setminus \O_{0})$ satisfying the inequality $\Delta_{p} u \geq \l u^{p-1}$ on $\mm^{m}_{\s} \setminus \bar \O_{0}$.  Next, choose $R \gg1$ in such a way that $\O_0 \Subset B_{R}(0)$. According to Theorem \ref{th-Lpmodel}, $u$ is $L^{p}$ on $\mm^{m}_{\s} \setminus B_{R}(0)$. It follows from Theorem \ref{th-lp-comparison} that $u \in W^{1,p}$ globally on $\mm^{m}_{\s} \setminus \bar B_{R}(0)$ thus proving the validity of $(\SR_{p})$.
\end{proof}

%%%%%%%%%%%%%%%%%%%%%%%%%%%%%%%%%%%%%%%%%%%%%%%%%%%%%%%%%%%%%%%%%%%%%%%%%%%%%%

\section{Compact support property}\label{section-compact}

In this final section, as an application of the theory  so far developed, we prove the compact support property stated in Theorem \ref{th-main3}. Actually, as we already did for Theorem \ref{th-main1}, we are going to prove a more general result.

\begin{theorem}\label{th-main3a}
 Let $\mm^{m}_{\s}$ be a $p$-stochastically complete and $p$-Feller model manifold, whose warping function satisfies
\begin{equation}\label{cs-laplacian}
\inf_{[0,+\infty)}\frac{\s'}{\s} > - \infty.
\end{equation}
Let $u \in C^{1}(\mm^{m}_{\s}\setminus \bar \O)$ be a nonnegative and bounded solution of
\[
 \Delta_{p} u \geq \Lambda(u)\, \text{on } \mm^{m}_{\s} \setminus \bar \O
\]
where $\Lambda : [0,+\infty) \to [0,+\infty)$ is a non-decreasing function satisfying
\[
(i)\, \Lambda\left(  0\right)  =0,\,\,\, (ii) \, \Lambda\left(  t\right) >0\text{, }\forall t>0, \,\,\, (iii) \liminf_{t\rightarrow0+}\frac {\Lambda\left(  t\right)  }{t^{\xi}}>0,
\]
for some $0\leq\xi< p-1$. Then $u$ has compact support.
\end{theorem}
\begin{remark}
 As it is pointed out in \cite{PuRiSe-jde} and  \cite[Example 1.35]{BMPR}, condition \eqref{cs-laplacian} cannot be avoided even in the linear setting. Note that, in that example, the model manifold is both stochastically complete and Feller for the Laplace-Beltrami operator.
\end{remark}
\begin{proof}
By Theorem \ref{th-main1a}, $u(x) \to 0$ as $x \to \infty$. Therefore, on noting that $\Lambda(t) \geq C t^{\xi}$  for all $0 \leq t <1$ with $t^{-\frac{\xi+1}{p}} \in L^{1}(0+)$, and recalling that $\Delta r = (m-1)\s'/\s$, we can apply \cite[Theorem 1.1]{PuRiSe-jde} to conclude that $u$ has compact support.
\end{proof}

%bibliography


\begin{thebibliography}{9999999}

\bibitem[Az]{Az} Azencott, Robert, {\it Behavior of diffusion semi-groups at infinity}.
Bull. Soc. Math. France {\bf 102} (1974), 193--240.
\bibitem[BMPR]{BMPR} Bianchini B.; Mari L.; Pucci P.; Rigoli M., \textit{On the interplay among maximum principles, compact support principles and Keller-Osserman conditions on manifolds}. Preprint 2018, available at \url{https://arxiv.org/pdf/1801.02102v1.pdf}
\bibitem[BPS]{BPS-Spectral} Bessa, G.P.; Pigola, S.; Setti, A.G. {\it Spectral and stochastic properties of the $f$-Laplacian, solutions of PDEs at infinity and geometric applications}. Rev. Mat. Iberoam. \textbf{29} (2013), no. 2, 579--610.
\bibitem[BS]{BS} Bianchi, D.; Setti A.G., \textit{Laplacian cut-offs, porous and fast diffusion on manifolds and other applications.} Calc. Var. Partial Differential Equations 57(1) (2018): 4.
\bibitem[DWZ]{DWZ-Advances} Dai, X.; Wei, G.; Zhang, Z., \textit{Local Sobolev constant estimate for integral Ricci curvature bounds.}  Adv. Math. \textbf{325} (2018), 1--33.
\bibitem[Gr]{Gr} Grigor'yan, A., \textit{Analytic and geometric background of recurrence and non-explosion of the Brownian motion on Riemannian manifolds}. Bulletin of Amer. Math. Soc. \textbf{36} (1999) 135--249.
\bibitem[GZ]{GaZi} Gariepy, R.; Ziemer, W. P., \textit{A regularity condition at the boundary for solutions of quasilinear elliptic equations.} Arch. Rational Mech. Anal. \textbf{67} (1977), 25--39
\bibitem[GM]{GM-JMPA} Grigor'yan, A.; Masamune, J. {\it Parabolicity and stochastic completeness of manifolds in terms of the Green formula}. J. Math. Pures Appl. (9) {\bf 100} (2013), no. 5, 607--632.
\bibitem[Ha]{Ha-lecturenotes} Hajlasz, P., \textit{Sobolev spaces and calculus of variations}. Lecture notes of the course held in Helsinki (1997). Available at \url{http://www.pitt.edu/~hajlasz/Notatki/hel-97.pdf}
\bibitem[Li]{Li-nonlinear}Lieberman, G.M.. \textit{Boundary regularity for solutions of degenerate elliptic equations}. Nonlinear Anal. 12 (1988), no. 11, 1203--1219.
\bibitem[MP]{MP} Mari, L.; Pessoa, L.F., \textit{Duality between Ahlfors-Liouville and Khas'minskii properties for nonlinear equations}. Comm. Anal. Geom. (to appear). Preprint available at \url{https://arxiv.org/pdf/1603.09113.pdf}.
\bibitem[MV]{MV-tams} Mari, L.; Valtorta, D., \textit{On the equivalence of stochastic completeness and Liouville and Khas'minskii conditions in linear and nonlinear settings}.  Trans. Amer. Math. Soc. \textbf{365} (2013), no. 9, 4699--4727.
\bibitem[PeW1]{PW-GAFA} P. Petersen, G. Wei, \textit{Relative volume comparison with integral curvature bounds.} Geom. Funct. Anal. \textbf{7} (1997), no. 6, 1031–-1045.
\bibitem[PeW2]{PW-TAMS}P. Petersen, G. Wei, \textit{Analysis and geometry on manifolds with integral Ricci curvature bounds. II.} Trans. Amer. Math. Soc. \textbf{353} (2001), no. 2, 457-–478.
\bibitem [PiSe1] {PS-Feller} Pigola, S.; Setti, A.G., \textit{The Feller property on Riemannian manifolds}. J. Funct. Anal. \textbf{262} (2012), no. 5, 2481--2515.
\bibitem [PiSe2] {PS-Ensaios} Pigola, S.; Setti, A.G. Global divergence theorems in nonlinear PDEs and geometry. Ensaios Matemáticos [Mathematical Surveys], {\bf 26}. Sociedade Brasileira de Matem\'atica, Rio de Janeiro, 2014.
\bibitem[PuRiSe]{PuRiSe-jde} Pucci, P.; Rigoli, M.; Serrin, J., {\it Qualitative properties for solutions of singular elliptic inequalities on complete manifolds}. J. Differential Equations {\bf 234} (2007), 507--543.
\bibitem[PuSe]{PuSe}  Pucci, P.; Serrin, J., \textit{The maximum principle. Progress in Nonlinear Differential Equations and their Applications}, 73. Birkhäuser Verlag, Basel, 2007.
\bibitem[PRS1]{PRS-stoch}  Pigola, S.; Rigoli, M.; Setti, A.G. \textit{A remark on the maximum principle and stochastic completeness}. Proc. Amer. Math. Soc. \textbf{131} (2003), no. 4, 1283--1288.
\bibitem[PRS2]{PRS-revista} Pigola, S.; Rigoli, M.; Setti, A.G., \textit{Some nonlinear  function theoretic properties of Riemannian manifolds}. Rev. Mat. Iberoam. \textbf{22} (2006), no. 3, 801--831.
\bibitem[RSV]{RSV-manuscripta} Rigoli, M.; Salvatori, M.; Vignati, M., \textit{A note on $p$-subharmonic functions on complete manifolds}. Manuscripta Math. \textbf{92} (1997), 339--359.
\bibitem[St]{St-book} Struwe, M., \textit{Variational methods}. Second edition. Ergebnisse der Mathematik und ihrer Grenzgebiete (3) \textbf{34}. Springer-Verlag, Berlin, 1996.
\bibitem[Sa1] {Sa-Aspects} Saloff-Coste, L., \textit{ Aspects of Sobolev-type inequalities.} London Mathematical Society Lecture Note Series 289. CUP, Cambridge  2002
J. Differential Geom. \textbf{36} (1992), 417--450.
\bibitem[SY]{SY} Schoen, R.M.; Yau, S.-T. , Lectures on differential geometry. Internetional press Cambridge, 1994.
\bibitem[To]{To} Tolksdorf, P., \textit{Regularity for a more general class of quasilinear elliptic equations.} J. Differential Equations \textbf{51} (1984), 126--150.
\bibitem[Tr]{Tr} Troyanov M., \textit{Parabolicity of Manifolds}, Siberian Advances in Mathematics {\bf 9} (1999) 125--150.

%%%%




\end{thebibliography}
\end{document}